\documentclass[a4paper,11pt]{article}
\usepackage{amsfonts,amssymb,amsthm,cite,amsmath,amstext}
\usepackage{color}
\usepackage[colorlinks,linkcolor=red,citecolor=blue]{hyperref}

\title{\textbf{Characterizing volume via cone duality}}
\author{\textsc{Jian Xiao}}
\date{}
\begin{document}
\maketitle

\theoremstyle{definition}
\newtheorem*{pf}{Proof}
\newtheorem{theorem}{Theorem}[section]
\newtheorem{remark}{Remark}[section]
\newtheorem{problem}{Problem}[section]
\newtheorem{conjecture}{Conjecture}[section]
\newtheorem{lemma}{Lemma}[section]
\newtheorem{corollary}{Corollary}[section]
\newtheorem{definition}{Definition}[section]
\newtheorem{proposition}{Proposition}[section]
\newtheorem{example}{Example}[section]

\begin{abstract}
For divisors over smooth projective varieties, we show that the volume can be characterized by the duality between pseudo-effective cone of divisors and movable cone of
curves. Inspired by this result, we give and study a natural intersection-theoretic volume functional for 1-cycles over compact K\"ahler manifolds.
In particular, for numerical equivalence classes of curves over projective varieties, it is closely related to the mobility functional.
\end{abstract}

\tableofcontents

\section{Introduction}
In this paper, all projective varieties are defined over $\mathbb{C}$. The volume of a divisor on projective variety is a non-negative number measuring the positivity of the divisor. Let $X$ be an $n$-dimensional smooth projective variety, and let $D$ be a divisor on $X$. By definition, the volume of $D$ is defined to be
$$vol(D):= \underset{m\rightarrow \infty}{\limsup}{\frac{h^0 (X, mD)}{m^n /n!}}.$$
Thus $vol(D)$ measures the asymptotic growth of the dimensions of the section space of multiplied divisors $mD$. We call $D$ a big divisor if $h^0 (X, mD)$ has growth of order $m^n$ as $m$ tends to infinity, that is, $D$ is big if and only if $vol(D)>0$. The pseudo-effective cone of divisors (denoted by $\mathcal{E}_{NS}$) is the closure of the cone generated by numerical classes of big divisors. It contains the cone of ample divisors as a subcone. It is well known that the volume $vol$ depends only on the numerical class of the divisor, and $vol^{1/n}$ is homogeneous of degree one, concave on the pseudo-effective cone and extends to a continuous function on the whole real N\'{e}ron-Severi space which is strictly positive exactly on big classes.

In the analytical context,
from the work \cite{Bou02a,Bou02b}, we know that the volume can be characterized by Monge-Amp\`{e}re mass, and from the work \cite{Dem10}, it can even be characterized by
Morse type integrals.
In this paper, the starting point is to give a new characterization of the volume of divisors by using cone duality. From the seminal work of Boucksom-Demailly-Paun-Peternell (see \cite{BDPP13}), we know the duality of the pseudo-effective cone of divisors and the cone generated by movable curves, that is,
$\mathcal{E}_{NS} ^{\vee} = \overline{\mathcal{M}}_{NS}$.
Using this cone duality and an invariant of movable curve, we give the following new volume characterization of divisors by the infimum of intersection
numbers between the pairings of $\mathcal{E}_{NS}$ and ${\mathcal{M}}_{NS}$.

\begin{theorem}
\label{thm divisor volume}
Let $X$ be an $n$-dimensional smooth projective variety and
let $\alpha\in N^1(X, \mathbb{R})$ be a numerical class of divisor. Then the volume of $\alpha$ can be
characterized as following:
\begin{equation*}
vol(\alpha) = \underset
{{\gamma \in {\mathcal{M}}_{NS,1}} }{inf}
 max (\langle \alpha, \gamma \rangle , 0)^n
\end{equation*}
where
${\mathcal{M}}_{NS,1}$ is a subset of the movable cone ${\mathcal{M}}_{NS}$ (see Definition \ref{def vol hat 1}).
Conversely, this volume characterization implies
the cone duality
$\mathcal{E}_{NS} ^{\vee} = \overline{\mathcal{M}}_{NS}$.
Furthermore, we can also replace the movable cone ${\mathcal{M}}_{NS}$ by the Gauduchon cone $\mathcal{G}$ or balanced cone $\mathcal{B}$ which is generated by special hermitian metrics.
\end{theorem}

\begin{remark}
\label{rmk vol kahler}
Under the conjecture on weak transcendental holomorphic Morse inequalities (see \cite{BDPP13}), the above result also holds true for any Bott-Chern $(1,1)$-class over compact K\"ahler manifolds. In particular, even without this assumption, for any $\alpha\in H^{1,1}_{BC}(X, \mathbb{R})$ over a hyper-K\"ahler manifold $X$,
we have $$vol(\alpha) = \underset
{{\gamma \in {\mathcal{M}}_1} }{inf}
 max (\langle \alpha, \gamma \rangle , 0)^n.$$
\end{remark}

Inspired by the above volume characterization for divisors, using cone dualities, we introduce a volume functional for 1-cycles over compact K\"ahler manifolds.
For smooth projective variety, by Kleiman's criterion, we have the cone duality $Amp^{\vee}=\overline{NE}$ where $Amp$ is the ample cone generated by ample divisors and $\overline{NE}$ is the cone generated by irreducible curves. For $n$-dimensional compact K\"ahler manifold, by Demailly-Paun's numerical characterization of K\"ahler cone (see \cite{DP04}), we have the cone duality $\mathcal{K}^{\vee}=\mathcal{N}$ where $\mathcal{K}$ is the K\"ahler cone generated by K\"ahler classes and $\mathcal{N}$ is the cone  generated by $d$-closed positive $(n-1,n-1)$-currents.

\begin{definition}
(1) Let $X$ be an $n$-dimensional smooth projective variety, and let $\gamma\in N_1(X, \mathbb{R})$ be a numerical equivalence class of curve. Let $Amp_1$ be the set containing all numerical classes of ample divisors of volume one. Then the volume of $\gamma$ is defined to be $$\widehat{vol}_{\overline{NE}} (\gamma) =\underset
{\beta \in Amp_1}{inf} max (\langle \beta, \gamma \rangle, 0)^{\frac{n}{n-1}}.$$
\\
(2) Let $X$ be an $n$-dimensional compact K\"ahler manifold, and let $\gamma\in H^{n-1, n-1}_{BC}(X, \mathbb{R})$ be a Bott-Chern $(n-1,n-1)$-class. Let $\mathcal{K}_1$ be the set containing all K\"ahler classes of volume one. Then the volume of $\gamma$ is defined to be $$\widehat{vol}_{\mathcal{N}} (\gamma) =\underset
{\gamma \in {\mathcal{K}_1}} {inf} max (\langle \beta, \gamma \rangle,0)^{\frac{n}{n-1}}.$$
\end{definition}

From its definition, it is clear $\widehat{vol}_{\overline{NE}}$ (resp. $\widehat{vol}_{\mathcal{N}}$) has concave property. It also has other nice properties.

\begin{theorem}
\label{theorem vol_N}
Let $X$ be an $n$-dimensional smooth projective variety (resp. compact K\"ahler manifold). Then
$\widehat{vol}_{\overline{NE}}$ (resp. $\widehat{vol}_{\mathcal{N}}$) is a continuous function on the whole vector space $N_1(X, \mathbb{R})$ (resp. $H^{n-1,n-1}_{BC}(X, \mathbb{R})$). Furthermore, $\gamma \in \overline{NE}^{\circ}$ (resp. $\mathcal{N} ^\circ$)
if and only if
 $\widehat{vol}_{\overline{NE}}(\gamma)>0$ (resp. $\widehat{vol}_{\mathcal{N}}(\gamma)>0$).
\end{theorem}

The functional $\widehat{vol}_{\overline{NE}}$ is closely related to the mobility functional recently introduced by Lehmann (see \cite {Leh13}).
Mobility functional for cycles was suggested in \cite{DELV11} as an analogue of the volume function for divisors. The motivation is that one can interpret the volume of a divisor $D$ as an asymptotic measurement of the number of general points contained in members of $|mD|$ as $m$ tends to infinity. Let $\gamma$ be a numerical equivalence class of $k$-cycles over an $n$-dimensional integral projective variety $X$, following \cite{DELV11}, Lehmann defined the mobility of $\gamma$ as following:
$$ mob(\gamma):= \underset{m\rightarrow \infty}{\limsup} \frac{mc(m\gamma)}{m^{\frac{n}{n-k}}/n!},$$
where $mc(m\gamma)$ is the mobility count of the cycle class $m\gamma$, which is the maximal non-negative integer $b$ such that any $b$ general points of $X$ are contained in a cycle of class $m\gamma$.
In particular, we can define the mobility for numerical classes of curves. Lehmann proved that the mobility functional also distinguishes
interior points and boundary points. Thus, in the situation of curves, combining with Theorem \ref{theorem vol_N}, we have two functionals with this property. It is interesting to compare $mob$ and $\widehat{vol}_{\overline{NE}}$ over $\overline{NE}$. The optimistical expectation is that there are two positive constants $c_1, c_2$ depending only on the dimension of the underlying manifold such that
$$c_1 \widehat{vol}_{\overline{NE}}(\gamma)\leq mob(\gamma)\leq c_2 \widehat{vol}_{\overline{NE}}(\gamma)$$
for any $\gamma \in \overline{NE}$. Moreover, we expect $\widehat{vol}_{\overline{NE}}(\gamma)=mob(\gamma)$. In this paper, we obtain the positive constant $c_2$ by using Lehamnn's estimates of mobility count functional $mc$. In a subsequent joint work \cite{LX15} with Lehmann, besides other results, we will obtain the positive constant $c_1$.
Indeed, for any fixed ample divisor $A$ and boundary point $\gamma\in \partial \overline{NE}$, it is not hard to obtain the asymptotic behaviour of the quotient $mob(\gamma+ \varepsilon A^{n-1})/ \widehat{vol}_{\overline{NE}}(\gamma+ \varepsilon A^{n-1})$ as $\varepsilon$ tends to zero.

\begin{theorem}
\label{theorem compare mob and vol_N}
Let $X$ be an $n$-dimensional smooth projective variety, and let $\overline{NE}$ be the closure of the cone generated by effective 1-cycles. Then for any $\gamma\in \overline{NE}$, we have $$mob(\gamma)\leq n! 2^{4n+1} \widehat{vol}_{\overline{NE}}(\gamma).$$
And for any fixed ample divisor $A$ and boundary point $\gamma\in \partial \overline{NE}$, there is a positive constant $c(A, \gamma)$ such that
$mob(\gamma+ \varepsilon A^{n-1})\geq c(A, \gamma)\varepsilon \widehat{vol}_{\overline{NE}}(\gamma+ \varepsilon A^{n-1})$. In particular, we have $$\underset{\varepsilon\rightarrow 0}{\liminf}\frac{mob(\gamma+ \varepsilon A^{n-1})}{\varepsilon \widehat{vol}_{\overline{NE}}(\gamma+ \varepsilon A^{n-1})} \geq c(A, \gamma).$$
\end{theorem}

With respect to our volume functional $\widehat{vol}_{\mathcal{N}}$, we want to study Fujita type approximation results for 1-cycles over compact K\"ahler manifolds. In this paper, following Boucksom's analytical version of divisorial Zariski decomposition (see \cite{Bou04}, \cite{Bou02b}) (for the algebraic approach, see \cite{Nak04}), we study Zariski decomposition for 1-cycles. In divisorial Zariski decomposition, the negative part is an effective divisor of Kodaira dimension zero, and indeed it contains only one positive $(1,1)$-current. In our setting, we can prove this fact also holds for big 1-cycles. Comparing with other definitions of Zariski decomposition for 1-cycles (see e.g. \cite{FL13}), effectiveness of the negative part is one of its advantage. Using his characterization of volume by Monge-Amp\`{e}re mass, Boucksom showed that the Zariski projection preserves volume. It is also expected that in our setting the Zariski projection preserves
$\widehat{vol}_{\mathcal{N}}$. Indeed, this follows from our another kind of Zariski decomposition for 1-cycles developed in \cite{LX15}, which is more closely related to $\widehat{vol}_{\mathcal{N}}$.

\begin{theorem}
\label{thm introduction zariski}
Let $X$ be an $n$-dimensional compact K\"ahler manifold and let $\gamma\in \mathcal{N}^\circ$ be an interior point. Let $\gamma=Z(\gamma)+\{N(\gamma)\}$ be the Zariski decomposition in the sense of Boucksom, then $N(\gamma)$ is an effective curve and it is the unique positive current contained in the negative part $\{N(\gamma)\}$.
As a consequence, this implies $\widehat{vol}_{\mathcal{N}}(\{N(\gamma)\})=0$. Moreover, we have $\widehat{vol}_{\mathcal{N}}(\gamma)
=\widehat{vol}_{\mathcal{N}}(Z(\gamma))$.
\end{theorem}

\section{Characterizing volume for divisors}

\subsection{Technical preliminaries}

\subsubsection{Smoothing movable classes}

Besides the well known cone duality $\mathcal{E}_{NS} ^\vee =\overline{\mathcal{M}}_{NS}$, we also have cone dualities between the cone defined by positive currents and the cone defined by positive forms. They provide a method to smooth movable classes, which will be useful in volume characterization by using special metrics.

Let $X$ be an $n$-dimensional compact complex manifold, then we have Bott-Chern cohomology groups $H^{\bullet, \bullet}_{BC}(X, \mathbb{K})$ and Aeppli cohomology groups $H^{\bullet, \bullet}_{A}(X, \mathbb{K})$ with $\mathbb{K}=\mathbb{R}$ or $\mathbb{C}$. Recall that we have canonical duality between $H^{\bullet, \bullet}_{BC}(X, \mathbb{K})$ and $H^{n-\bullet, n-\bullet}_{A}(X, \mathbb{K})$ (see \cite{AT13}).

\begin{definition}
\label{def cone}
Let $X$ be an $n$-dimensional compact complex manifold.
\\
(1) The cone $\mathcal{E}$ is defined to be the convex cone in $H^{1,1}_{BC}(X, \mathbb{R})$ generated by $d$-closed positive $(1,1)$-currents;
\\
(2) The cone $\mathcal{E}_A$ is defined to be the convex cone in $H^{1,1}_{A}(X, \mathbb{R})$ generated by $dd^c$-closed positive $(1,1)$-currents;
\\
(3) The balanced cone $\mathcal{B}$ is defined to be the convex cone in $H^{n-1,n-1}_{BC}(X, \mathbb{R})$ generated by $d$-closed strictly positive $(n-1,n-1)$-forms;
\\
(4) The Gauduchon cone $\mathcal{G}$ is defined to be the convex cone in $H^{n-1,n-1}_{A}(X, \mathbb{R})$ generated by $dd^c$-closed strictly positive $(n-1,n-1)$-forms.
\end{definition}

Under the duality of $H^{1,1}_{BC}(X, \mathbb{R})$ and $H^{n-1,n-1}_{A}(X, \mathbb{R})$ and the duality of $H^{1,1}_{A}(X, \mathbb{R})$ and $H^{n-1,n-1}_{BC}(X, \mathbb{R})$, we have the following cone dualities between the above positive cones.

\begin{proposition}
\label{prop cone dual}
Let $X$ be an $n$-dimensional compact complex manifold, then we have
$\mathcal{E}^{\vee}= \overline{\mathcal{G}}$ and $\mathcal{E}_A ^{\vee}= \overline{\mathcal{B}}$
\end{proposition}

\begin{proof}
Indeed, the above cone dualities are consequences of geometric Hahn-Banach theorem, for example, one can see \cite{Sul76}, \cite{Lam99} or \cite{Tom10}. For reader's convenience, let us sketch its proof. Firstly, we prove $\mathcal{E}^{\vee}= \overline{\mathcal{G}}$. Let $\alpha$ be a real smooth $(1,1)$-form. Applying Lamari's characterization of
positive $(1,1)$-currents, we know that there exists a distribution $\psi$ such that $\alpha+dd^c \psi$ is a positive $(1,1)$-current if and only if
$$\int \alpha \wedge G \geq 0$$
for any $dd^c$-closed strictly positive $(n-1,n-1)$-form $G$ (thus $G=\omega^{n-1}$ for some Gauduchon metric $\omega$). Under the natural duality of $H^{1,1}_{BC}(X, \mathbb{R})$ and $H^{n-1,n-1}_{A}(X, \mathbb{R})$, it is clear this implies the cone duality $\mathcal{E}^{\vee}= \overline{\mathcal{G}}$.
Using the same technique (Hahn-Banach theorem), one can also give a characterization of $dd^c$-closed positive $(1,1)$-currents.
More precisely, there exists a $(0,1)$-current $\theta$ such that $\alpha+\partial\theta+ \overline{\partial\theta}$ is a positive $(1,1)$-current if and only if
$$\int \alpha \wedge B \geq 0$$
for any $d$-closed strictly positive $(n-1,n-1)$-form $B$ (thus $B=\omega^{n-1}$ for some balanced metric $\omega$). Under the natural duality of $H^{1,1}_{A}(X, \mathbb{R})$ and $H^{n-1,n-1}_{BC}(X, \mathbb{R})$, this implies the cone duality $\mathcal{E}_A ^{\vee}= \overline{\mathcal{B}}$.
\end{proof}

Recall that the cone of movable curves $\mathcal{M}_{NS}$ is generated by numerical equivalence classes of curves of the form $\mu_* (\tilde{A}_1 \wedge...\wedge \tilde{A}_{n-1})$, where $\mu: \tilde{X}\rightarrow X$ ranges among all modifications with $\tilde{X}$ smooth projective and $\tilde{A}_1,...,\tilde{A}_{n-1}$ range among all ample divisors over $\tilde{X}$.
And its transcendental version is the movable cone
$\mathcal{M}\subseteq H^{n-1,n-1}_{BC}(X, \mathbb{R})$ over a compact K\"ahler manifold $X$. $\mathcal{M}$ is the cone generated by all the Bott-Chern classes of the form $[\mu_* (\tilde{\omega}_1 \wedge...\wedge \tilde{\omega}_{n-1})]_{BC}$, where $\mu: \tilde{X}\rightarrow X$ ranges among all modifications with $\tilde{X}$ K\"ahler and $\tilde{\omega}_1,...,\tilde{\omega}_{n-1}$ range among all K\"ahler metrics over
$\tilde{X}$.

Our first observation is that any current $\mu_* (\tilde{\omega}_1 \wedge...\wedge \tilde{\omega}_{n-1})$ can be smoothed to be a Gauduchon metric $G$ such that $[\mu_* (\tilde{\omega}_1 \wedge...\wedge \tilde{\omega}_{n-1})]_A =[G]_A$.

\begin{proposition}
\label{prop smoothing gauduchon}
Let $\mu: \tilde{X}\rightarrow X$ be a modification between compact complex manifold, and let $\tilde G$ be a Gauduchon metric on $\tilde{X}$. Then $\mu_* \tilde G$ can be smoothed to be a Gauduchon metric $G$ such that $[\mu_* \tilde G]_A=[G]_A$.
\end{proposition}

\begin{proof}
From the cone duality $\mathcal{E}^{\vee}= \overline{\mathcal{G}}$, in order to prove $[\mu_* \tilde G]_A \in \mathcal{G}$, we only need to verify that $[\mu_* \tilde G]_A$ is an interior point of $\mathcal{E}^{\vee}$$(=\overline{\mathcal{G}})$. For any $\alpha\in \mathcal{E}\setminus \{[0]_{BC}\}$, since the pull-back $\mu^* \alpha$ is also pseudo-effective, we have $$\langle[\mu_* \tilde G]_A, \alpha\rangle =
\langle[\tilde G]_A, \mu^* \alpha\rangle \geq 0.$$
Take a positive current $\tilde T \in \mu^* \alpha$, then we have
$$\langle[\tilde G]_A, \mu^* \alpha\rangle=
\int \tilde G \wedge \tilde T .$$
By the strictly positivity of $\tilde G$, $\int \tilde G \wedge \tilde T =0$ if and only if $\tilde T =0$, and this contradicts to our assumption $\alpha=[\mu_* \tilde T]_{BC}\in \mathcal{E}\setminus \{[0]_{BC}\}$. Thus $\langle[\mu_* \tilde G]_A, \alpha\rangle >0$ for any $\alpha\in \mathcal{E}\setminus \{[0]_{BC}\}$, and this implies $[\mu_* \tilde G]_A$ is an interior point of $\overline{\mathcal{G}}$, which means that there exists a Gauduchon metric $G$ such that $[\mu_* \tilde G]_A= [G]_A$.
\end{proof}

Indeed, the current $\mu_* (\tilde{\omega}_1 \wedge...\wedge \tilde{\omega}_{n-1})$ can not only be smoothed to be a Gauduchon class, it can also smoothed to be a balanced metric $B$ such that $[\mu_* (\tilde{\omega}_1 \wedge...\wedge \tilde{\omega}_{n-1})]_{BC} =[B]_{BC}$.
From the proof of Proposition \ref{prop smoothing gauduchon}, we see that a key ingredient is that the pull-back of cohomology class in $\mathcal{E}$ contains positive currents. Analogue to this fact, due to a result of \cite{AB95}, one can also always pull back Aeppli class in $\mathcal{E}_A$ and get $dd^c$-closed positive $(1,1)$-currents on the manifold upstairs.

\begin{lemma}(see \cite{AB95})
\label{lemma AB95}
Let $\mu: \tilde{X}\rightarrow X$ be a modification between compact complex manifold, and let $T$ be a $dd^c$-closed positive $(1,1)$-current on $X$. Then there exists an unique $dd^c$-closed positive $(1,1)$-current $\tilde T \in \mu^*[T]_A$ such that
$\mu_* \tilde T=T$.
\end{lemma}

We remark that the above fact is already used by Toma (see \cite{Tom10}), and the following proposition is essentially due to Toma.

\begin{proposition}
\label{prop smoothing balanced}
Let $\mu: \tilde{X}\rightarrow X$ be a modification between compact balanced manifold, and let $\tilde B$ be a balanced metric on $\tilde{X}$. Then $\mu_* \tilde B$ can be smoothed to be a balanced metric $B$ such that $[\mu_* \tilde B]_{BC}=[B]_{BC}$.
\end{proposition}

\begin{proof}
Similar to the proof of Proposition \ref{prop smoothing gauduchon}, we only need to show
$$\langle[\mu_* \tilde B]_{BC}, \alpha\rangle >0$$
for any $\alpha\in \mathcal{E}_A\setminus \{[0]_{A}\}$. Now, by using Lemma \ref{lemma AB95}, for any $\alpha=[T]_A \in \mathcal{E}_A\setminus \{[0]_{A}\}$, one can find a non-zero $dd^c$-closed positive $(1,1)$-current $\tilde T \in \mu^*\alpha$, so we have $$\langle[\mu_* \tilde B]_{BC}, \alpha\rangle=
 \langle[\tilde B]_{BC}, \mu^*\alpha\rangle=\int \tilde B \wedge \tilde T >0.$$
And as a consequence, there exists a balanced metric $B$ such that $[\mu_* \tilde B]_{BC}=[B]_{BC}$.
\end{proof}

\subsubsection{An invariant of movable classes}

In this subsection, we introduce an (universal) invariant $\mathfrak{M}$ for movable, balanced or Gauduchon classes. This invariant is defined by cone duality and intersection numbers. We will see that they coincide with the volume of K\"ahler classes if the cohomology classes are given by the $(n-1)$-power of K\"ahler classes.

\begin{definition}
\label{def hat vol movable}
Let $X$ be an $n$-dimensional compact K\"ahler manifold, and let $\gamma$ be a movable (or balanced, or Gauduchon) class. Let $\mathcal{E}_1$ be the set of pseudo-effective classes of volume one. Then
the invariant $\mathfrak{M}(\gamma)$ is defined as following:
$$\mathfrak{M}(\gamma) :=\underset{\beta\in \mathcal{E}_1}{inf}\langle\beta, \gamma\rangle ^{\frac{n}{n-1}} .$$
\end{definition}

\begin{remark}
\label{rmk vol hat proj}
In the case when $X$ is a smooth projective variety, we can also define $\mathfrak{M}(\gamma)$ for $\gamma\in \mathcal{M}_{NS}$. In this situation, the parings $\langle\beta, \gamma\rangle$ are the parings of numerical equivalence classes of divisors and curves.
\end{remark}

\begin{remark}
\label{rmk cone coincide}
Recall that we have the cone dualities $\mathcal{E}^{\vee}=\overline{\mathcal{G}}$ and $\mathcal{E}_A ^{\vee}=\overline{\mathcal{B}}$. Indeed, under the assumption of the conjectured transcendental cone duality $\mathcal{E}^{\vee}=\overline{\mathcal{M}}$ (see \cite{BDPP13}), the movable cone $\mathcal{M}$, the balanced cone $\mathcal{B}$ and the Gauduchon cone $\mathcal{G}$ should be the same, that is, $\mathcal{E}^{\vee}=\overline{\mathcal{M}}=\overline{\mathcal{B}}
=\overline{\mathcal{G}}$ (see e.g. \cite{FX14}). This is why we call $\mathfrak{M}$ is an universal invariant associated to  movable, balanced and Gauduchon classes over compact K\"ahler manifolds.
\end{remark}

It is clear that, from its definition, we have
$$\mathfrak{M}(\gamma_1 + \gamma_2)^{\frac{n-1}{n}} \geq \mathfrak{M}(\gamma_1)^{\frac{n-1}{n}}+ \mathfrak{M}( \gamma_2)^{\frac{n-1}{n}}.$$

\begin{proposition}
\label{prop vol hat coin kahler}
Let $X$ be an $n$-dimensional compact K\"ahler manifold, and let $\gamma=\omega^{n-1}$ for some K\"ahler class $\omega$, then we have
$\mathfrak{M}(\gamma)=vol(\omega)$.
\end{proposition}

\begin{proof}
Firstly, let $\beta=\frac{\omega}{vol(\omega)^{\frac{1}{n}}}$, it is clear that $$\langle \beta, \omega^{n-1}\rangle=vol(\omega)^{\frac{n-1}{n}}, $$
which implies $\mathfrak{M}(\gamma)\leq vol(\omega)$. On the other hand, we claim that, for any $\beta\in \mathcal{E}$ with $vol(\beta)=1$, we have
$$\langle\beta, \gamma\rangle ^{\frac{n}{n-1}}\geq vol(\omega).$$
This is just the Khovanskii-Teissier inequality which follows from the singular version of Calabi-Yau theorem (see \cite{Bou02a}): there exists a positive $(1,1)$-current $T\in \beta$ such that $$T_{ac}^n = \Phi$$
 almost everywhere, where $\Phi=\omega^n/{vol(\omega)}$ and $T_{ac}$ is the absolutely continuous part of $T$ with respect to Lebesgue measure. Here we use the same symbol $\omega$ to denote a K\"ahler metric in the K\"ahler class $\omega$. Then we have
 $$\langle\beta, \gamma\rangle=\int T\wedge \omega^{n-1}\geq
 \int T_{ac} \wedge \omega^{n-1}\geq \int (\frac{T_{ac}^n}{\Phi})^{\frac{1}{n}}
 (\frac{\omega^n}{\Phi})^{\frac{n-1}{n}}\Phi
 =vol(\omega) ^{\frac{n-1}{n}}.$$
This implies the claim, thus finishing our proof.
\end{proof}

An easy corollary is the strictly positivity of $\mathfrak{M}$.

\begin{corollary}
\label{cor positive vol hat}
Let $X$ be an $n$-dimensional compact K\"ahler manifold, and let $\gamma\in \mathcal{G}$ (resp. $\mathcal{M}$ or $\mathcal{B}$) be an interior point, then we have
$\mathfrak{M}(\gamma)>0$.
\end{corollary}

Next let $\mu:\tilde{X}\rightarrow X$ be a modification between compact K\"ahler manifolds, we want to study the behaviour of $\mathfrak{M}$
under $\mu$. Firstly, we need the following elementary fact on the transform of the volume of pseudo-effective $(1,1)$-classes under bimeromorphic maps.

\begin{lemma}
\label{lemma modification vol (1,1)}
Let $\mu:\tilde{X}\rightarrow X$ be a modification between $n$-dimensional compact K\"ahler manifolds. Assume $\beta\in \mathcal{E}$ is a pseudo-effective class on $X$, then $vol(\beta)=vol(\mu^* \beta)$; Assume $\tilde \beta\in \widetilde{\mathcal{E}}$ is a pseudo-effective class on $\tilde X$, then $vol(\tilde\beta)\leq vol(\mu_* \tilde\beta)$.
\end{lemma}

\begin{proof}
Recall that the volume of $\beta$ is defined to be the supremum of Monge-Amp\`{e}re mass, that is,
$$vol(\beta)=\underset{T}{sup} \int T_{ac}^n,$$
where $T$ ranges among all positive $(1,1)$-currents in the class $\beta$. For any positive current $T\in \beta$, we obtain a positive current $\mu^*T \in\mu^* \beta$. By the definition of the absolutely part with respect to Lebesgue measure, we have $(\mu^*T)_{ac}=\mu^*T_{ac}$. And $T_{ac}$ is a $(1,1)$-form with $L^1 _{loc}$ coefficients. In particular, analytic subset is of zero measure with respect to the measure $T_{ac}^n$, which yields
$$\int (\mu^*T)_{ac}^n =\int T_{ac}^n.$$
This implies $vol(\mu^* \beta)\geq vol(\beta)$. On the other hand, for any positive current $\tilde T \in\mu^* \beta$, we get a positive current $\mu_*\tilde T\in \beta$. By $(\mu_*\tilde T)_{ac} =\mu_*\tilde T_{ac}$, we obtain $vol(\mu^* \beta)\leq vol(\beta)$. All in all we have $vol(\beta)=vol(\mu^* \beta)$.
Similarly, it is also easy to see $vol(\tilde\beta)\leq vol(\mu_* \tilde\beta)$ for any $\tilde\beta \in \widetilde{\mathcal{E}}$.
\end{proof}

Now we can show that $\mathfrak{M}$ has the same property as $vol$ under bimeromorphic maps. We only state the result for K\"ahler manifolds. It is clear that $\mathfrak{M}$ admits an extension to the closure of $\mathcal{G}$ (resp. $\mathcal{M}$ or $\mathcal{B}$).

\begin{proposition}
\label{prop modification vol_ G B}
Let $\mu:\tilde{X}\rightarrow X$ be a modification between $n$-dimensional compact K\"ahler manifolds. Assume $\gamma\in \overline{\mathcal{G}}$ (resp. $\overline{\mathcal{M}}$ or $\overline{\mathcal{B}}$ ), then $\mathfrak{M}(\gamma)=\mathfrak{M}(\mu^* \gamma)$; Assume $\tilde \gamma\in \overline{\widetilde{\mathcal{G}}}$ (resp. $\overline{\widetilde{\mathcal{M}}}$ or $\overline{\widetilde{\mathcal{B}}}$), then $\mathfrak{M}(\tilde\gamma)\leq \mathfrak{M}(\mu_* \tilde\gamma)$.
\end{proposition}

\begin{proof}
We firstly consider the pull-back case. By Lemma
\ref{lemma modification vol (1,1)}, for any fixed
$\tilde\beta \in \tilde{\mathcal{E}}$
with $vol(\tilde\beta)=1$, we have
$$\langle \tilde\beta, \mu^*\gamma \rangle
= \langle \mu_*\tilde\beta, \gamma \rangle
\geq \langle \frac{\mu_*\tilde \beta}{vol(\mu_*\tilde \beta)^{1/n}}, \gamma \rangle
\geq \mathfrak{M}(\gamma)^{\frac{n-1}{n}}.$$
This clearly implies $\mathfrak{M}(\mu^*\gamma)\geq \mathfrak{M}(\gamma)$. For the other direction, for any fixed $\beta\in \mathcal{E}$ with $vol(\beta)=1$, using Lemma \ref{lemma modification vol (1,1)} again, we have
$$\langle\beta, \gamma\rangle
=\langle\mu_*(\mu^*\beta), \gamma\rangle
=\langle\mu^*\beta, \mu^*\gamma\rangle
\geq  \mathfrak{M}(\mu^*\gamma)^{\frac{n-1}{n}}.$$
Thus, $\mathfrak{M}(\mu^*\gamma)\leq \mathfrak{M}(\gamma)$, and as a consequence, we finish the proof of $\mathfrak{M}(\mu^*\gamma)= \mathfrak{M}(\gamma)$. For the push-forward case, the proof of
$\mathfrak{M}(\tilde\gamma)\leq \mathfrak{M}(\mu_* \tilde\gamma)$ is the same.
\end{proof}

We remark that the inequality $\mathfrak{M}(\tilde\gamma)\leq \mathfrak{M}(\mu_* \tilde\gamma)$ is important in the characterization of the volume of divisors in the following section.

\subsection{Volume characterization}

In this section, using the invariant $\mathfrak{M}$ introduced in the previous section, we show the volume of divisors can be characterized by cone duality.

\begin{definition}
\label{def vol hat 1}
Let $X$ be an $n$-dimensional smooth projective variety, and let $\mathcal{M}_{NS}$ be the cone of movable curves. Then $\mathcal{M}_{NS,1}$ is defined to be the subset containing all $\gamma\in \mathcal{M}_{NS}$ with $\mathfrak{M}(\gamma)=1$. Similarly, for compact K\"ahler manifolds, we can define
$\mathcal{G}_1, \mathcal{M}_1$ and $\mathcal{B}_1$ in the same way.
\end{definition}

\begin{theorem}
\label{thm characterize vol of divisor}
\label{big divisor}
Let $X$ be an $n$-dimensional smooth projective variety and
let $\alpha\in N^1(X, \mathbb{R})$ be a numerical equivalence class of divisor. Then the volume of $\alpha$ can be
characterized as following:
$$
(\star)\ vol(\alpha)^{\frac{1}{n}} = \underset{
{\gamma \in {\mathcal{M}}_{NS,1}} }{inf} max (\langle \alpha, \gamma \rangle, 0).
$$
Conversely, this volume characterization implies
the cone duality
$\mathcal{E}_{NS} ^{\vee} = \overline{\mathcal{M}}_{NS}$.
Moreover, we can also replace the cone of movable curves by Gauduchon or balanced cone, that is,
$$
vol(\alpha)^{\frac{1}{n}}
=\underset
{\gamma \in {\mathcal{G}_1}} {inf} max(\langle \alpha, \gamma \rangle,0)
=\underset
{\gamma \in {\mathcal{B}_1}} {inf} max(\langle \alpha, \gamma \rangle,0)
.
$$
\end{theorem}

\begin{proof}
We first consider the case when $\alpha$ is not pseudo-effective, by the definition of volume of divisors, it is clear $vol(\alpha)=0$. On the other hand, the cone duality $\mathcal{E}_{NS} ^\vee = \overline{\mathcal{M}}_{NS}$ implies there exists some interior point $\gamma\in \mathcal{M}_{NS}$ such that $\langle\alpha, \gamma\rangle<0$. Furthermore, using Corollary \ref{cor positive vol hat}, we can even normalize $\gamma$ such that $\mathfrak{M}(\gamma)=1$.
Thus $\underset{
{\gamma \in {\mathcal{M}}_{NS,1}} }{inf} max (\langle \alpha, \gamma \rangle, 0)=0=vol(\alpha)^{\frac{1}{n}}$.\\

Next consider the case when $\alpha$ is given by a big divisor.
By the very definition of $\mathfrak{M}$, for any $\gamma\in \mathcal{M}_{NS}$, it is clear that
$$
\langle\frac{\alpha}{vol(\alpha)^{1/n}}, \gamma\rangle
\geq \mathfrak{M}(\gamma)^{\frac{n-1}{n}},
$$
or equivalently,
$$
\langle\alpha, \gamma\rangle
\geq vol(\alpha)^{\frac{1}{n}}\mathfrak{M}(\gamma)^{\frac{n-1}{n}}.
$$
In particular, for any $\gamma\in \mathcal{M}_{NS,1}$, this yields
$
\langle\alpha, \gamma\rangle
\geq vol(\alpha)^{\frac{1}{n}}.
$
Thus we have
$$
vol(\alpha)^{\frac{1}{n}}\leq \underset
{\gamma \in {\mathcal{M}}_{NS,1}} {inf} \langle \alpha, \gamma \rangle.
$$
In order to prove the equality, we need to show that, for any $\varepsilon>0$, there exists a movable class $\gamma_\varepsilon\in \mathcal{M}_{NS,1}$ such that
$$
\langle \alpha, \gamma_\varepsilon\rangle \leq vol(\alpha)^{\frac{1}{n}}+\varepsilon.
$$
This mainly depends on approximating Zariski decomposition of K\"ahler currents and orthogonality estimates of the decomposition (see \cite{BDPP13}). Since $\alpha$ is given by a big divisor, for any $\delta>0$, there exists a modification $\mu_\delta :X_\delta \rightarrow X$ such that $\mu_\delta ^* \alpha =\beta_\delta +[E_\delta]$ with $\beta_\delta$ given by an ample divisor and $E_\delta$ given by an effective divisor. Moreover, we also have
\begin{equation}
\label{eq approximation zariski ineq}
vol(\alpha)-\delta \leq vol(\beta_\delta)\leq vol(\alpha)
\end{equation}
and
\begin{equation}
\label{eq orthogonality zariski ineq}
\langle [E_\delta], \beta_\delta ^{n-1}\rangle
\leq
c(vol(\alpha)-vol(\beta_\delta))^{1/2}
\end{equation}
where $c$ is a positive constant depending only on the class $\alpha$ and dimension $n$. Applying (\ref{eq approximation zariski ineq}) and (\ref{eq orthogonality zariski ineq}) to $\langle \alpha, \mu_{\delta*}\beta_\delta ^{n-1}\rangle$, we get
\begin{align}
\label{eq after orthogonality}
\langle \alpha, \mu_{\delta*}(\beta_\delta ^{n-1})\rangle
&= \langle \mu_\delta ^*\alpha, \beta_\delta ^{n-1}\rangle\\
&= vol(\beta_\delta)+\langle [E_\delta], \beta_\delta ^{n-1}\rangle\\
&\leq  vol(\alpha)+\mathbf{O}(\delta^{1/2}).
\end{align}
Next by Proposition \ref{prop vol hat coin kahler} and Proposition \ref{prop modification vol_ G B}, we know that
\begin{equation}
\label{eq push-forward}
\mathfrak{M}(\mu_{\delta*}(\beta_\delta ^{n-1}))\geq
\mathfrak{M}{(\beta_\delta ^{n-1})}=vol(\beta_\delta).
\end{equation}
We claim that $\gamma_\delta:={\mu_{\delta*}(\beta_\delta ^{n-1})}/{\mathfrak{M}(\mu_{\delta*}(\beta_\delta ^{n-1}))^{\frac{n-1}{n}}}$ is our desired movable class. Firstly, by the definition of $\mathfrak{M}$, it is obvious that $\mathfrak{M}(\gamma_\delta)=1$. Secondly, by using (\ref{eq approximation zariski ineq}) and (\ref{eq push-forward}), we can estimate
$\langle\alpha, \gamma_\delta\rangle$ as following:
\begin{align}
\langle\alpha, \gamma_\delta\rangle
&\leq \langle\mu_\delta ^*\alpha, \frac{\beta_\delta ^{n-1}}{vol(\beta_\delta)^{n-1/n}}\rangle\\
&\leq
vol(\alpha)^{1/n}+
[\frac{vol(\alpha)}{(vol(\alpha)-\delta)^{\frac{n-1}{n}}}
-vol(\alpha)^{1/n}]
+\mathbf{O}(\delta^{1/2})
\end{align}
Thus, for any $\varepsilon>0$, we can choose some $\delta(\varepsilon)>0$, such that $\gamma_{\delta(\varepsilon)}$ is our desired movable class. In summary, we have finished the proof of the equality $$vol(\alpha)^{\frac{1}{n}} = \underset
{\gamma \in {\mathcal{M}}_{NS,1}} {inf} \langle \alpha, \gamma \rangle$$
for big class $\alpha$.\\

In the case when $\alpha$ lies on the boundary of $\mathcal{E}_{NS}$, for any $\varepsilon>0$ and ample divisor $A$, apply the above proved equality for $\alpha+\varepsilon A$, we have
$$ vol(\alpha+\varepsilon A)^{\frac{1}{n}} = \underset
{\gamma \in {\mathcal{M}}_{NS,1}} {inf} \langle \alpha+\varepsilon A, \gamma \rangle.$$
Take ${inf}$ on both sides with respect to
${\varepsilon>0} $, we get the equality for boundary class.
\\

Now we show that $(\star)$ implies
$\mathcal{E}_{NS} ^\vee = \overline{\mathcal{M}}_{NS}$. It is obvious $\mathcal{E}_{NS}\subseteq \overline{\mathcal{M}}_{NS} ^\vee$. In order to prove the converse inclusion, we only need to show: if $\alpha$ is an interior point of $\overline{\mathcal{M}}_{NS} ^\vee$, then $\alpha$ is also an interior point of $\mathcal{E}_{NS}$ (or equivalently, $vol(\alpha)>0$). Fix an ample divisor $A$. Since $\alpha$ is an interior point of $\overline{\mathcal{M}}_{NS} ^\vee$, for $\varepsilon>0$ small, $\alpha-\varepsilon A$ also lies in the interior of $\overline{\mathcal{M}}_{NS} ^\vee$. In particular, we have $\langle\alpha, \gamma\rangle>\langle\varepsilon A, \gamma\rangle$ for any $\gamma\in \overline{\mathcal{M}}_{NS}\setminus [0]$. Then $(\star)$ implies
$$vol(\alpha)^{\frac{1}{n}}= \underset{
{\gamma \in {\mathcal{M}}_{NS,1}} }{inf} max (\langle \alpha, \gamma \rangle, 0)
\geq \varepsilon vol(A)^{\frac{1}{n}}>0.$$

For the volume characterization by Gauduchon or balanced cone, from the proof for movable cone, one can see that if we can show $\gamma_{\delta(\varepsilon)}$ can be smoothed to be a Gauduchon or balanced class, then we have the desired equality. And this just follows from the results of Proposition \ref{prop smoothing gauduchon} and Proposition \ref{prop smoothing balanced}.
\end{proof}

\begin{remark}
\label{rmk vol kahler conj}
Let $X$ be an $n$-dimensional compact K\"ahler manifold. Under the assumption of the conjectured weak transcendental holomorphic Morse inequality, that is, $$vol(\alpha-\beta)\geq \alpha^n-n\alpha^{n-1}\cdot\beta$$ for any nef classes $\alpha,\beta$ (for recent progress of this problem, one can see \cite{Xia13}, \cite{Pop14}), then we will also have orthogonality estimates for interior points of $\mathcal{E}$ (see \cite{BDPP13}). By the arguments above, we will have volume characterization for any Bott-Chern $(1,1)$-class $\alpha$, that is,
$$\ vol(\alpha)^{\frac{1}{n}} = \underset{
{\gamma \in {\mathcal{M}}_{1}} }{inf} max (\langle \alpha, \gamma \rangle, 0). $$
Moreover, this implies the cone duality $\mathcal{E}^\vee =\overline{\mathcal{M}}$. Thus it is natural to ask whether one can prove this volume characterization without using orthogonality estimates of approximation Zariski decomposition. And this also provides new perspectives to prove the conjectured cone duality $\mathcal{E}^\vee =\overline{\mathcal{M}}$.
\end{remark}

\section{Volume functional for 1-cycles}

\subsection{Definition and properties}
Inspired by Theorem \ref{thm characterize vol of divisor}, using cone dualities, we introduce a volume functional for the numerical equivalence class of curves over smooth projective varieties and a
volume functional for Bott-Chern $(n-1,n-1)$-classes over compact K\"ahler manifolds.
For smooth projective variety, we have the ample cone $Amp$ generated by ample divisors and the cone $\overline{NE}$ generated by irreducible curves. Then we have the cone duality $$Amp^{\vee}=\overline{NE}$$ which is just Kleiman's criterion. For $n$-dimensional compact K\"ahler manifold, we have K\"ahler cone $\mathcal{K}$ generated by K\"ahler classes and the cone $\mathcal{N}$ generated by $d$-closed positive $(n-1,n-1)$-currents. Then we have the cone duality $$\mathcal{K}^{\vee}=\mathcal{N}$$ which follows from Demailly-Paun's numerical characterization of K\"ahler cone (see \cite{DP04}).

Now we can give the following definition.
\begin{definition}
\label{def vol n ne}
(1) Let $X$ be an $n$-dimensional smooth projective variety, and let $\gamma\in N_1(X, \mathbb{R})$ be a numerical equivalence class of curve. Let $Amp_1$ be the set containing all numerical classes of ample divisors of volume one. Then the volume of $\gamma$ is defined to be $$\widehat{vol}_{\overline{NE}} (\gamma) =\underset
{\beta \in Amp_1}{inf} max (\langle \beta, \gamma \rangle, 0)^{\frac{n}{n-1}}.$$
\\
(2) Let $X$ be an $n$-dimensional compact K\"ahler manifold, and let $\gamma\in H^{n-1, n-1}_{BC}(X, \mathbb{R})$ be a Bott-Chern $(n-1,n-1)$-class. Let $\mathcal{K}_1$ be the set containing all K\"ahler classes of volume one. Then the volume of $\gamma$ is defined to be $$\widehat{vol}_{\mathcal{N}} (\gamma) =\underset
{\gamma \in {\mathcal{K}_1}} {inf} max (\langle \beta, \gamma \rangle,0)^{\frac{n}{n-1}}.$$
\end{definition}

\begin{remark}
\label{rmk compare movable}
In the volume characterization of divisors, using the cone duality $\mathcal{E}_{NS}^\vee = \overline{\mathcal{M}}_{NS}$ (or the conjectured $\mathcal{E}^\vee = \overline{\mathcal{M}}$), we introduce an invariant $\mathfrak{M}$ for movable classes (see Definition \ref{def hat vol movable}). Now $\widehat{vol}_{\mathcal{N}}$ gives another invariant of movable classes when it restricts on $\overline{\mathcal{M}}$. From their definitions, it is clear we have $\mathfrak{M}(\gamma)\leq \widehat{vol}_{\mathcal{N}}(\gamma)$ for any $\gamma\in \overline{\mathcal{M}}$. Unlike $vol$ giving an uniform volume functional on $\mathcal{E}$ and $\overline{\mathcal{K}}$, \emph{we do not know whether they would coincide on the movable cone}. In general, the nef cone $\overline{\mathcal{K}}$ can be strictly contained in $\mathcal{E}$, it seems possible that $\mathfrak{M}$ may be smaller than $\widehat{vol}_{\mathcal{N}}$. However, if $X$ is a projective or compact K\"ahler surface,
both our volume functional $\widehat{vol}_{\mathcal{N}}$ (or $\widehat{vol}_{\overline{NE}}$) and $\mathfrak{M}$ coincide with the usual volume for pseudo-effective classes.
\end{remark}

\begin{example}
\label{eg vol}
To illustrate the definition of volume functional for 1-cycles, we propose to do some concrete calculations on an example similar to the one due to Cutkosky \cite{Cut86} (it is also contained in \cite{Bou04}). Let $Y$ be a smooth projective surface, and let $D, H$ be two very ample divisors over $Y$. Let $X=\mathbb{P}(\mathcal{O}(D)\oplus \mathcal{O}(-H))$ with its canonical projection $\pi: X \rightarrow Y$. Denote by $L=\mathcal{O}_X (1)$ the tautological bundle of $X$, then the nef cone $\mathcal{K}_X$ of $X$ is generated by $\pi^* \mathcal{K}_Y$ and $\pi^*H +L$. In Cutkosky's example, $Y$ is an Abelian surface (or more generally, a projective surface with $\overline{\mathcal{K}}_Y=\mathcal{E}_Y$). For simplicity, we consider the very simple case $Y=\mathbb{P}^2$ with $D=\mathcal{O}(d), H=\mathcal{O}(1)$, then we have $$L^3=(d-1)^2 +d, \pi^*H^2 \cdot L=1, \pi^*H \cdot L^2=d-1, \pi^*H^3=0.$$ Let $\alpha= a \pi^*H +b (\pi^*H +L)$ with $a, b\in \mathbb{R}_{+}$ be a nef class, then the volume of $\alpha$ is as following
$$vol(\alpha)=b^3 ((d-1)^2)+d)+ 3b^2(a+b)(d-1)+3(a+b)^2 b.$$
Consider the 1-cycle $\gamma(x,y)= x \pi^*H^2 +y \pi^*H\cdot L$ with $x, y\geq 0$, then we have
$$\langle\alpha, \gamma(x,y)\rangle= (a+b)y+bx+by(d-1).$$
From the above expressions, we have an explicit formula of $\widehat{vol}_{\overline{NE}}$. In particular, if we take $d=1$, then
$$\widehat{vol}_{\overline{NE}}(\gamma(x,y))=\underset{\underset{a,b\geq 0}{b^3+3(a+b)^2 b=1}}{inf}(by+(a+b)x)^\frac{3}{2}.$$
\end{example}

The volume functionals $\widehat{vol}_{\mathcal{N}}$ and
$\widehat{vol}_{\overline{NE}}$ have many nice properties. For simplicity, we only state the result for $\widehat{vol}_{\mathcal{N}}$. The argument for $\widehat{vol}_{\overline{NE}}$ is similar.

\begin{theorem}
\label{thm vol_N distinguish}
Let $X$ be an $n$-dimensional compact K\"ahler manifold. Then $\widehat{vol}_{\mathcal{N}}$ has the following properties:\\
(1) $\widehat{vol}_{\mathcal{N}}^{\frac{n-1}{n}}$ is concave and homogeneous of degree one.\\
(2) $\widehat{vol}_{\mathcal{N}}$ is continuous on the whole vector space $H^{n-1, n-1}_{BC}(X, \mathbb{R})$.\\
(3) $\gamma \in \mathcal{N}^{\circ}$ if and only if $\widehat{vol}_{\mathcal{N}}(\gamma)>0$.
\end{theorem}

\begin{proof}
Property (1) just follows from the definition of $\widehat{vol}_{\mathcal{N}}$.
Now let us first prove property (3). Let $\gamma\in \mathcal{N}^{\circ}$ be an interior point, we want to show that $\widehat{vol}_{\mathcal{N}}(\gamma)>0$. $\gamma\in \mathcal{N}^{\circ}$ means that there exists some K\"ahler class $\omega$ such that $\gamma-\omega^{n-1} \in \mathcal{N}$, this implies
$\widehat{vol}_{\mathcal{N}}(\gamma)\geq \widehat{vol}_{\mathcal{N}}(\omega^{n-1})$.
We claim that $$\widehat{vol}_{\mathcal{N}}(\omega^{n-1})=vol(\omega),$$
which yields $\widehat{vol}_{\mathcal{N}}(\gamma)\geq vol(\omega)>0$.  The proof of this claim is the same with Proposition \ref{prop vol hat coin kahler}, so we omit it.
Conversely, we need to show that if $\widehat{vol}_{\mathcal{N}}(\gamma)>0$ then $\gamma\in \mathcal{N}^{\circ}$. Otherwise, $\gamma\in \partial\mathcal{N}\setminus \{[0]_{BC}\}$. And the cone duality $\mathcal{K}^\vee =\mathcal{N}$ implies there exists some $\theta\in \overline{\mathcal{K}}\setminus \{[0]_{BC}\}$ such that
$\langle\theta, \gamma\rangle=0$. Fix a K\"ahler class $\omega$. For any $\varepsilon>0$, we consider the K\"ahler class $\theta+\varepsilon\omega$ and the following intersection number $$\rho_\varepsilon :=\langle\frac{\theta+\varepsilon\omega}
{vol(\theta+\varepsilon\omega)^{1/n}}, \gamma\rangle.$$
Since $\theta\in \overline{\mathcal{K}}\setminus \{[0]_{BC}\}$, the class $\theta$ contains at least one non-zero positive current, then we have $\langle\theta, \omega^{n-1}\rangle>0$. And we have
$$vol(\theta+\varepsilon\omega)^{\frac{1}{n}}\geq n\langle\theta, \omega^{n-1}\rangle \varepsilon^{\frac{n-1}{n}}=\mathbf{O}(\varepsilon^{\frac{n-1}{n}}).$$
Using $\langle\theta, \gamma\rangle=0$, we get $\rho_\varepsilon \leq \mathbf{O}(\varepsilon^{1/n})$. Thus, $\widehat{vol}_{\mathcal{N}}(\gamma)=0$. In conclusion, we have proved that $\gamma \in \mathcal{N}^{\circ}$ if and only if $\widehat{vol}_{\mathcal{N}}(\gamma)>0$.

Next we consider the continuity of $\widehat{vol}_{\mathcal{N}}$, thus proving property (2). Since concave function defined in a convex set is continuous in the interior. In order to show the continuity of $\widehat{vol}_{\mathcal{N}}$, we need to verify
$$\underset{\varepsilon\rightarrow 0}{\lim }\widehat{vol}_{\mathcal{N}}(\gamma+\varepsilon \omega^{n-1})=0$$
for any $\gamma\in \partial\mathcal{N}\setminus \{[0]_{BC}\}$ and any K\"ahler class $\omega$. Indeed, for $\gamma\in \partial\mathcal{N}\setminus \{[0]_{BC}\}$, we will prove
\begin{align}
\label{eq vol_N boundary behaviour}
\widehat{vol}_{\mathcal{N}}(\gamma+\varepsilon \omega^{n-1})\leq \mathbf{O}(\varepsilon^{\frac{1}{n-1}}).
\end{align}
The arguments are similar with the estimation of $\rho_\varepsilon$, but with little modification. Once again, using the fact $\gamma\in \partial\mathcal{N}\setminus \{[0]_{BC}\}$, there exists some $\theta\in \overline{\mathcal{K}}\setminus \{[0]_{BC}\}$ such that
$\langle\theta, \gamma\rangle=0$. We consider the following intersection number
\begin{align}
\label{eq vol_N boundary behaviour rho}
\rho_{\delta,\varepsilon} :=\langle\frac{\theta+\delta\omega}
{vol(\theta+\delta\omega)^{1/n}}, \gamma+\varepsilon \omega^{n-1}\rangle
\end{align}
with $\delta$ positive to be determined. Using $\langle\theta, \gamma\rangle=0 $ and $\langle\theta, \omega^{n-1}\rangle>0$ again, it is easy to see that
\begin{align}
\label{eq vol_N boundary behaviour rho1}
\rho_{\delta,\varepsilon}
\leq \mathbf{O}(\delta^{\frac{1}{n}}+\delta^{\frac{1}{n}}\varepsilon+
\delta^{-\frac{n-1}{n}}
\varepsilon).
\end{align}
Take $\delta=\varepsilon$, we get $\rho_{\delta,\varepsilon}\leq\mathbf{O}(\varepsilon^{1/n})$, which implies
\begin{align}
\label{eq vol_N boundary behaviour1}
\widehat{vol}_{\mathcal{N}}(\gamma+\varepsilon \omega^{n-1})\leq \mathbf{O}(\varepsilon^{\frac{1}{n-1}}),
\end{align}
thus finishing the proof of continuity.
\end{proof}

We give a new interpretation of our volume functional as the infinimum of a family of geometric norms. We only work for $\widehat{vol}_{\mathcal{N}}$, and the arguments go through mutatis mutandis for the volume functional $\widehat{vol}_{\overline{NE}}$.

\begin{lemma}
\label{geometric norm}
(see also Corollary 2.8 of \cite{FL13}) Let $X$ be an $n$-dimensional compact K\"ahler manifold. Then any K\"ahler class $\alpha$ gives a norm $||\cdot ||_{\alpha}$ over $H^{n-1,n-1}_{BC}(X, \mathbb{R})$. Moreover, for $\gamma\in \mathcal{N}$, we have $||\gamma||_{\alpha}=\langle\alpha, \gamma\rangle$.
\end{lemma}

\begin{proof}
For any fixed K\"ahler class $\alpha$, there exist $d=h^{1,1}$ K\"ahler classes $\alpha_1,..., \alpha_d$ such that $\alpha_1,..., \alpha_d$ constitute a basis of the real vector space $H^{1,1}_{BC}(X, \mathbb{R})$, and $\alpha=\underset{1\leq i\leq d}{\sum} \alpha_i$. Then for any $\eta\in H^{n-1,n-1}_{BC}(X, \mathbb{R})$, we define $||\eta ||_{\alpha}$ as following:
$$||\eta ||_{\alpha}=\underset{1\leq i\leq d}{\sum}|\langle\alpha_i, \eta\rangle|.$$
It is clear that the above $||\cdot ||_{\alpha}$ is a norm, since it is just the sum of absolute values of the coordinates with respect to the basis $\alpha_1,..., \alpha_d$. Now, for $\gamma\in \mathcal{N}$, we have $\langle\alpha_i, \gamma\rangle\geq 0$. And this implies $$||\gamma||_{\alpha}=\underset{1\leq i\leq d}{\sum}\langle\alpha_i, \gamma\rangle=\langle\alpha, \gamma\rangle.$$
\end{proof}

Now by the definition of $\widehat{vol}_{\mathcal{N}}$, we have the following proposition.

\begin{proposition}
\label{prop geom norm}
Let $X$ be an $n$-dimensional compact K\"ahler manifold, then for $\gamma\in \mathcal{N}$ we have $$\widehat{vol}_{\mathcal{N}}^{\frac{n-1}{n}}(\gamma)=\underset{\alpha\in \mathcal{K}_1}{inf} ||\gamma||_{\alpha}.$$
\end{proposition}

\subsection{Relation with mobility}

In this section, we focus on comparing $\widehat{vol}_{\overline{NE}}$ and Lehmann's mobility functional $mob$ for 1-cycles over smooth projective variety.
Firstly, let us recall the definition of mobility of numerical equivalence classes of curves. Let $\gamma$ be a 1-cycle class over $X$ of dimension $n$, the mobility of $\gamma$ is defined as following:
$$ mob(\gamma):= \underset{m\rightarrow \infty}{\limsup} \frac{mc(m\gamma)}{m^{\frac{n}{n-1}}/n!},$$
where $mc(m\gamma)$ is the mobility count of the 1-cycle class $m\gamma$ defined as the maximal non-negative integer $b$ such that any $b$ general points of $X$ are contained in a 1-cycle of class $m\gamma$.
From Theorem \ref{thm vol_N distinguish} and Theorem A in \cite{Leh13}, both functionals take positive values exactly in $\overline{NE}^{\circ}$ and are continuous over $\overline{NE}$. Moreover, both of them are homogeneous over $\overline{NE}$, it is natural to propose the following question.

\begin{conjecture}\label{conj mob volne}
Let $X$ be a smooth projective variety, then $mob=\widehat{vol}_{\overline{NE}}$, or at least there exist two positive constants
$c_1$ and $c_2$ depending only on the dimension of $X$ such that $$c_1\widehat{vol}_{\overline{NE}}\leq mob\leq c_2\widehat{vol}_{\overline{NE}}.$$
\end{conjecture}

We observe that the constant $c_2$ is provided by the upper bound estimation of mobility count. For any fixed ample divisor $A$ and boundary point $\gamma\in \partial \overline{NE}$, it is clear that if we can find a positive constant $c(A, \gamma)$ such that
$$
\underset{\varepsilon\rightarrow 0}{\liminf}\frac{mob(\gamma+ \varepsilon A^{n-1})}{\widehat{vol}_{\overline{NE}}(\gamma+ \varepsilon A^{n-1})} \geq c(A, \gamma),
$$
then we can obtain the desired uniform constant $c_1$. In this direction, we can get a weaker asymptotic behaviour as $\varepsilon$ tends to zero.

\begin{theorem}
\label{thm compare mob and vol_NE}
Let $X$ be an $n$-dimensional smooth projective variety. Then for any $\gamma\in \overline{NE}$, we have $$mob(\gamma)\leq n! 2^{4n+1} \widehat{vol}_{\overline{NE}}(\gamma).$$
And for any fixed ample divisor $A$ and boundary point $\gamma\in \partial \overline{NE}$, there is a positive constant $c(A, \gamma)$ such that
$$mob(\gamma+ \varepsilon A^{n-1})\geq c(A, \gamma)\varepsilon \widehat{vol}_{\overline{NE}}(\gamma+ \varepsilon A^{n-1}).$$
In particular, we have $$\underset{\varepsilon\rightarrow 0}{\liminf}\frac{mob(\gamma+ \varepsilon A^{n-1})}{\varepsilon \widehat{vol}_{\overline{NE}}(\gamma+ \varepsilon A^{n-1})} \geq c(A, \gamma).$$
\end{theorem}

\begin{proof}
The upper bound $c_2$ relies on the estimations of mobility counts. By homogeneity and continuity, we only need to consider the case when $\gamma$ is given by a 1-cycle with $\mathbb{Z}$-coefficients. We need Lehmann's upper bound estimation (see Theorem 6.24 of \cite{Leh13}):  let $A$ be a very ample divisor and let $s$ be a positive integer such that $\langle A, \gamma\rangle\leq s vol(A)$, then
$$mc(\gamma)\leq 2^{4n+1} s^{\frac{n}{n-1}}vol(A).$$
Indeed, by inspection of the proof of Theorem 6.24 of \cite{Leh13}, any real number $s\geq 1$ is sufficient for the above estimation of $mc(\gamma)$. Fix a $\mathbb{Q}$-ample divisor $\alpha$, then there exists a positive integer $m_\alpha$ such that $m_\alpha \alpha$ is very ample. And for this very ample divisor $m_\alpha \alpha$, there exists a positive integer $k_\alpha$ such that
\begin{align}
\frac{\langle m_\alpha \alpha,k \gamma\rangle}{vol(m_\alpha \alpha)}
\geq 1
\end{align}
for all positive integer $k\geq k_\alpha$
Applying Lehmann's mobility count estimation to $k\gamma$ when $A=m_\alpha \alpha$ and $s=\frac{\langle m_\alpha \alpha,k \gamma\rangle}{vol(m_\alpha \alpha)}$, we get
\begin{align}
mc(k \gamma)&\leq 2^{4n+1}\big(\frac{\langle m_\alpha \alpha,k \gamma\rangle}{vol(m_\alpha \alpha)}\big)^{\frac{n}{n-1}}
vol(m_\alpha \alpha)\\
&= 2^{4n+1}\langle\frac{\alpha}{vol(\alpha)^{1/n}}, k\gamma\rangle^{\frac{n}{n-1}}.
\end{align}
This yields the upper bound of $mob(\gamma)$:
\begin{align}
mob(\gamma)=\underset{k\rightarrow \infty}{\limsup}
\frac{mc(k \gamma)}{k^{n/n-1}/n!}\leq n!2^{4n+1}\langle\frac{\alpha}{vol(\alpha)^{1/n}}, \gamma\rangle^{\frac{n}{n-1}}.
\end{align}
Since any point of the ample cone can be approximated by $\mathbb{Q}$-ample divisors, we obtain our desired
\begin{align}
mob(\gamma)\leq n! 2^{4n+1} \widehat{vol}_{\overline{NE}}(\gamma).
\end{align}

Now let us consider the lower bound.
In the proof of Theorem \ref{thm vol_N distinguish} (see (\ref{eq vol_N boundary behaviour})-(\ref{eq vol_N boundary behaviour1})), we obtain the estimation of $\widehat{vol}_{\mathcal{N}}(\gamma+ \varepsilon \omega^{n-1})$. Using similar argument, we can get the same estimation of $\widehat{vol}_{\overline{NE}}(\gamma+ \varepsilon A^{n-1})$ with $\gamma\in \partial \overline{NE}$ and $A$ ample, that is,
\begin{align}
\label{eq vol_NE boundary behaviour}
\widehat{vol}_{\overline{NE}}(\gamma+ \varepsilon A^{n-1})\leq \mathbf{O}(\varepsilon^{\frac{1}{n-1}}).
\end{align}
By the basic property of mobility functional (see Lemma 6.17 of \cite{Leh13}), we have
\begin{align}
\label{eq mob boundary behaviour}
mob(\gamma+ \varepsilon A^{n-1})\geq mob(\varepsilon A^{n-1})=\mathbf{O}(\varepsilon^{\frac{n}{n-1}}).
\end{align}
Thus we get
$$mob(\gamma+ \varepsilon A^{n-1})\geq
c(A, \gamma)\varepsilon \widehat{vol}_{\overline{NE}}(\gamma+ \varepsilon A^{n-1})$$
for some positive constant $c(A, \gamma)$.
In particular, we have $$\underset{\varepsilon\rightarrow 0}{\liminf}\frac{mob(\gamma+ \varepsilon A^{n-1})}{\varepsilon \widehat{vol}_{\overline{NE}}(\gamma+ \varepsilon A^{n-1})} \geq c(A, \gamma).$$
\end{proof}

\begin{remark}
In order to obtain such an uniform lower bound $c_1$, what we expect is a better estimation of $mob(\gamma+ \varepsilon A^{n-1})$, that is,
$$mob(\gamma+ \varepsilon A^{n-1})\geq \mathbf{O}(\varepsilon^{\frac{1}{n-1}}),$$
as $\varepsilon$ tends to zero.
To obtain this, we need deeper understanding of $mob$.
\end{remark}

\begin{remark}
Just from its definition, the mobility functional $mob$ seems very hard to compute. For example, even in the case of complete intersection of ample divisor (see Question 7.1 of \cite{Leh13}), we do not know how to calculate its mobility. However, using our
volume functional, we have seen that $\widehat{vol}_{\overline{NE}}(A^{n-1})=vol(A)$ for any ample divisor $A$. For the concavity of $mob$, it is conjectured (see Conjecture 6.20 of \cite{Leh13}) that
$$mob(\gamma_1 +\gamma_2)^{\frac{n-1}{n}}\geq mob(\gamma_1 )^{\frac{n-1}{n}}+mob(\gamma_2)^{\frac{n-1}{n}}.$$
For our $\widehat{vol}_{\overline{NE}}$, concavity just follows from its definition (see Theorem \ref{thm vol_N distinguish}). Thus concavity of $mob$ will follow if we can prove $mob=\widehat{vol}_{\overline{NE}}$.
\end{remark}

\subsection{Towards Fujita approximation for 1-cycles}

In the work of \cite{FL13}, Fulger and Lehmann proved the existence of Zariski decomposition for big cycles with respect to mobility functional. Moreover, they also proved a Fujita type approximation for numerical class of curves. Our goal is to give such a Fujita type approximation for Bott-Chern classes of $d$-closed positive $(n-1,n-1)$-currents over compact K\"ahler manifolds with respect to our volume functional $\widehat{vol}_{\mathcal{N}}$, thus also give a Fujita type approximation for numerical class of curves over projective variety with respect to $\widehat{vol}_{\overline{NE}}$. Analogue to Fujita approximation for K\"ahler currents (see inequality (\ref{eq approximation zariski ineq})), one may conjecture the following:
\\

\emph{Let $X$ be an $n$-dimensional compact K\"ahler manifold and let $\gamma\in \mathcal{N}^\circ$. Then for any $\varepsilon>0$, there exists a proper modification $\mu: \widetilde{X}\rightarrow X$ with $\widetilde{X}$ K\"ahler such that $\mu^* \gamma= \beta_\varepsilon+[C_\varepsilon]$ and $\widehat{vol}_{\mathcal{N}}(\gamma)-\varepsilon\leq \widehat{vol}_{\mathcal{N}}(\beta_\varepsilon)\leq \widehat{vol}_{\mathcal{N}}(\gamma)$, where $\beta_\varepsilon$ is an interior point of movable cone ${\widetilde{\mathcal{M}}}$ (or balanced cone ${\widetilde{\mathcal{B}}}$) and $C_\varepsilon$ is an effective curve.}
\\

Indeed, if we have the decomposition $\mu^* \gamma= \beta_\varepsilon+[C_\varepsilon]$, then we have $\gamma-\mu_* \beta_\varepsilon \in \mathcal{N}$. This implies $\widehat{vol}_{\mathcal{N}}(\mu_* \beta_\varepsilon)\leq \widehat{vol}_{\mathcal{N}}(\gamma)$. Now similar to Proposition \ref{prop modification vol_ G B}, it is easy to see $\widehat{vol}_{\mathcal{N}}( \beta_\varepsilon)\leq\widehat{vol}_{\mathcal{N}}(\mu_* \beta_\varepsilon)$. Thus the above expected decomposition automatically implies $\widehat{vol}_{\mathcal{N}}( \beta_\varepsilon)\leq \widehat{vol}_{\mathcal{N}}(\gamma)$. Unfortunately, the pull-back $\mu^* \gamma$ need not to be a pseudo-effective class in general. Note that $\mu^* \gamma$ is pseudo-effective over $\widetilde{X}$ if and only if $\langle \mu^* \gamma, \tilde{\alpha}\rangle \geq 0$ for any K\"ahler class $\tilde{\alpha}$, which is equivalent to $\langle \gamma, \mu_* \tilde{\alpha}\rangle \geq 0$. In general, $\mu_* \tilde{\alpha}$ is not a nef class on $X$. By the cone duality $\overline{\mathcal{K}}^ \vee = \mathcal{N}$, we have $\langle \gamma, \mu_* \tilde{\alpha}\rangle <0$ if $\mu_* \tilde{\alpha}\notin \overline{\mathcal{K}}$. Anyhow, if $\gamma\in \overline{\mathcal{M}}$ is movable, then its pull-back $\mu^* \gamma$ is also movable (thus pseudo-effective). For movable classes, it is possible to obtain the conjectured decomposition
$\mu^* \gamma= \beta_\varepsilon+[C_\varepsilon]$ with desired properties.

To prove Fujita approximation for $\gamma$ with respect to our volume functional, the first step of our strategy is to decompose $\gamma$ over the underlying manifold $X$ into some ``good" part with its volume near the volume of $\gamma$. We also call it the positive part, and call the difference the negative part. Here ``good" means we can find a positive current in the class with less singularities, then we may get a movable or balanced class from its pull-back on some K\"ahler manifold $\widetilde{X}$ such that its volume is as near $\widehat{vol}_{\mathcal{N}}(\gamma)$ as possible (this will be developed in our subsequent work \cite{LX15}). Besides the desired positive part, we also want to obtain some effective curve from such a decomposition. For the Zariski decomposition of Fulger and Lehmann, in general the negative part is not the class of an effective curve (see Example 5.18 of \cite{FL13}). In the work \cite{Bou04}, Boucksom defined a beautiful divisorial Zariski decomposition for any pseudo-effective $(1,1)$-class over compact complex manifolds. Boucksom's definition is totally analytic which depends on Siu decomposition of positive currents (see \cite{Siu74}). And it can be seen as a cohomology version of Siu decomposition. As Siu decomposition holds for $d$-closed positive currents of any bidegree, the method of Boucksom provides a possible Zariski decomposition for pseudo-effective $(n-1,n-1)$-classes. However, unlike the $(1,1)$-classes, we do not have an analogue of Demailly's regularization theorem (see \cite{Dem92}) for $d$-closed positive $(n-1,n-1)$-currents.
We know little about the singularities of such currents. Thus we can not expect too much about such decompositions.
Following Boucksom's method of divisorial Zariski decomposition, we give such a decomposition for pseudo-effective $(n-1,n-1)$-classes. It shares many nice properties with divisorial Zariski decomposition.

Firstly, we give the definition of minimal multiplicity.

\begin{definition}
\label{def minimal multi}
Let $X$ be an $n$-dimensional compact K\"ahler manifold with a K\"ahler metric $\omega$, and let $\gamma\in \mathcal{N}$ be a pseudo-effective $(n-1,n-1)$-class.\\
(1) The minimal multiplicity of $\gamma$ at the point $x$ is defined to be
$$
\nu(\gamma,x):=\underset{\varepsilon>0}{sup} \  \underset{T_\varepsilon}{inf}\ \nu(T_\varepsilon, x),
$$
where $T_\varepsilon \in \gamma$ ranges among all currents such that $T_\varepsilon\geq -\varepsilon\omega^{n-1}$ (we also denote this set by  $\gamma[-\varepsilon\omega^{n-1}]$) and $\nu(T_\varepsilon, x)$ is the Lelong number of $T_\varepsilon$ at $x$.\\
(2) For any irreducible curve $C$, the minimal multiplicity of $\gamma$ along $C$ is defined to be
$$\nu(\gamma,C):=\underset{x\in C}{inf}\ {\nu(\gamma, x)}.$$
\end{definition}

\begin{remark}
\label{rmk multi}
It is easy to see that $\nu(\gamma,x)$ is finite. And $\nu(\gamma,C)=\nu(\gamma,x)$ for a generic point $x\in C$, here generic means outside at most countable union of analytic subsets.
\end{remark}

\begin{definition}
\label{def zariski }
Let $\gamma\in \mathcal{N}$ be a pseudo-effective $(n-1,n-1)$-class, the negative part $N(\gamma)$ of $\gamma$ is defined to be
$N(\gamma):=\sum \nu(\gamma, C)[C]$,
where $C$ ranges among all irreducible curves on $X$. And the positive part $Z(\gamma)$ of $\gamma$ is defined to be
$
Z(\gamma):=\gamma-\{N(\gamma)\}
$. And we call $\gamma=Z(\gamma)+\{N(\gamma)\}$ the Zariski decomposition of $\gamma$.
\end{definition}

Intuitively, the positive part $Z(\gamma)$ should share almost all positivity of $\gamma$ and the negative part should have very little positivity. Indeed, in the divisorial Zariski decomposition case, using his volume characterization by Monge-Amp\`{e}re mass, Boucksom
showed that $vol(\alpha)=vol(Z(\alpha))$ for any $\alpha\in \mathcal{E}$ over compact K\"ahler manifolds. In our setting, one way to compare the positivity of $Z(\gamma)$ and $\gamma$ is to compare their respective volumes $\widehat{vol}_{\mathcal{N}}(Z(\gamma))$ and $\widehat{vol}_{\mathcal{N}}(\gamma)$. For the negative part $\{N(\gamma)\}$, like the one in divisorial Zariski decomposition, we find $N(\gamma)$ is an effective curve which is very rigidly embedded in $X$ if we assume $\gamma$ is an interior point. This is an advantage compared with the other decompositions (e.g. the decompositions in \cite{FL13} and \cite{LX15}).

\begin{remark}
From Theorem \ref{thm vol_N distinguish}, it is clear that $\widehat{vol}_{\mathcal{N}}(\gamma)
=\widehat{vol}_{\mathcal{N}}(Z(\gamma))=0$ if $\gamma\in \partial\mathcal{N}$. And by the concavity of $\widehat{vol}_{\mathcal{N}}$, the equality $\widehat{vol}_{\mathcal{N}}(\gamma)
=\widehat{vol}_{\mathcal{N}}(Z(\gamma))$
will imply $\widehat{vol}_{\mathcal{N}}(\{N(\gamma)\})=0$.
\end{remark}

\begin{theorem}
\label{thm zariski}
Let $X$ be an $n$-dimensional compact K\"ahler manifold and let $\gamma\in \mathcal{N}^\circ$ be an interior point. Let $\gamma=Z(\gamma)+\{N(\gamma)\}$ be the Zariski decomposition in the sense of Boucksom, then $N(\gamma)$ is an effective curve and it is the unique positive current contained in the negative part $\{N(\gamma)\}$.
As a consequence, this implies $\widehat{vol}_{\mathcal{N}}(\{N(\gamma)\})=0$. Moreover, we have $\widehat{vol}_{\mathcal{N}}(\gamma)
=\widehat{vol}_{\mathcal{N}}(Z(\gamma))$.
\end{theorem}

\begin{proof}
We first prove the first part of the above theorem. Indeed, as the Zariski decomposition here is an $(n-1,n-1)$-analogue of Boucksom's divisorial Zariski decomposition, the statement concerning $N(\gamma)$ can be proved using almost the same arguments as in \cite{Bou04}. In \cite{Bou04}, some arguments use Demailly's regularization theorem. As we do not have such a regularization theorem for $(n-1,n-1)$-currents, for reader's convenience, we present the details here. The assumption $\gamma\in \mathcal{N}^\circ$ will play the role as Demailly's regularization theorem in the divisorial Zariski decomposition situation.

We first show the claim ($\ast$): $N(\gamma)=\sum \nu(\gamma, C)[C]$ is the unique positive current in the class $\{N(\gamma)\}$ if $\gamma\in \mathcal{N}^\circ$.
We remark that claim ($\ast$) implies $\widehat{vol}_{\mathcal{N}}(\{N(\gamma)\})=0$ (or equivalently, $\{N(\gamma)\}\in \partial\mathcal{N}$). Otherwise, $\{N(\gamma)\}\in \mathcal{N}^\circ$. Fix a K\"ahler class $\omega$, then there exists a positive constant  $\delta>0$ such that $\{N(\gamma)\}-\delta\omega^{n-1} \in \mathcal{N}^\circ$. In particular, there exists a positive current $\Theta \in \{N(\gamma)\}$ such that $\Theta\geq \delta\omega^{n-1}$. Here we use the same symbol $\omega$ to represent a K\"ahler metric in the K\"ahler class $\omega$. Note that $H^{n-1,n-1}_A (X, \mathbb{R})\neq \{[0]_A\}$ over compact K\"ahler manifolds, thus there exists some smooth $(n-2, n-2)$-form $\psi$ such that $i\partial\bar\partial \psi\neq 0$. For $\varepsilon>0$ small enough, $$\Theta_\varepsilon := \Theta+\varepsilon i\partial\bar\partial \psi \in \{N(\gamma)\}$$
is a positive current and $\Theta_\varepsilon\neq \Theta$, contradicting our claim ($\ast$).

Now let us begin the proof of the claim ($\ast$). The proof is divided into several steps.

\begin{lemma}
\label{lemma multiplicity}
Let $\gamma\in \mathcal{N}^\circ$, then
$\nu(\gamma, C)=\underset{0\leq T\in \gamma}{inf} \nu(T,C)$
for any irreducible curve $C$.
\end{lemma}

\begin{proof}
To prove this, we only need to verify
$\nu(\gamma, x)=\underset{0\leq T\in \gamma}{inf} \nu(T,x)$
for any point $x$, then we will have $$\nu(\gamma, C)=\underset{x\in C}{inf}\nu(\gamma, x)=\underset{x\in C}{inf}\underset{0\leq T\in \gamma}{inf} \nu(T,x)
=\underset{0\leq T\in \gamma}{inf} \nu(T,C).$$

From the definition of $\nu(\gamma,x)$, we only need to prove
\begin{align}
\nu(\gamma,x)\geq \underset{0\leq T \in \gamma}{inf}\nu(T, x).
\end{align}
As $\gamma\in \mathcal{N}^{\circ}$, there exists a positive current $T\in \gamma$ such that $T\geq \beta^{n-1}$ for some K\"ahler metric $\beta$. Fix $\varepsilon>0$, for any $\delta>0$ there exists a current $T_{\varepsilon, \delta}\in \gamma[-\varepsilon\beta^{n-1}]$ such that
\begin{align}
\nu(T_{\varepsilon, \delta}, x)-\delta<\underset{T_\varepsilon}{inf}\nu(T_{\varepsilon}, x),
\end{align}
where $T_{\varepsilon}$ ranges among
$\gamma[-\varepsilon\beta^{n-1}] $.
Since $T\geq \beta^{n-1}$, we have $(1-\varepsilon)T_{\varepsilon,\delta}+ \varepsilon T\geq \varepsilon^2 \beta^{n-1}
$
which is a positive current in $\gamma$, thus
\begin{align}
\underset{0\leq T\in \gamma}{inf} \nu(T,x)&\leq
\nu((1-\varepsilon)T_{\varepsilon,\delta}+ \varepsilon T,x)\\
&\leq
(1-\varepsilon)\underset{T_\varepsilon }{inf}{\nu(T_\varepsilon, x)}
+(1-\varepsilon)\delta+ \varepsilon\nu(T,x).
\end{align}
Now let $\delta\rightarrow 0$ and then let $\varepsilon\rightarrow 0$, we get the desired inequality
$\nu(\gamma, x) \geq \underset{0\leq T\in \gamma}{inf} \nu(T,x)$.
\end{proof}

\begin{lemma}
\label{lemma big zariski projection}
(compare with Proposition 3.8 of \cite{Bou04}) Let $\gamma\in \mathcal{N}^\circ$, then $Z(\gamma)\in \mathcal{N}^\circ$ and
$\nu(Z(\gamma), C)=0$.
\end{lemma}

\begin{proof}
Once again, $\gamma\in \mathcal{N}^\circ$ implies there exists a positive current $T\in \gamma$ such that $T\geq \beta^{n-1}$ for some K\"ahler metric $\beta$. Apply Siu decomposition to the $d$-closed positive current $T-\beta^{n-1}$:
$$T-\beta^{n-1}=R+ \sum \nu(T-\beta^{n-1}, C)[C]=R+ \sum \nu(T, C)[C]$$
 for some residue positive current $R$. Then the definition of $N(\gamma)$ implies $$T-\beta^{n-1}-N(\gamma)\geq 0,$$
  which yields $T-N(\gamma)\geq \beta^{n-1}$. This implies $Z(\gamma)=\{T-N(\gamma)\}\in \mathcal{N}^\circ$. Indeed, by the above arguments, Siu decomposition also shows that any positive current in $Z(\gamma)$ are of the form $T-N(\gamma)$ for some positive current $T\in \gamma$. With Lemma \ref{lemma multiplicity} and this fact, we get
\begin{align}
\nu(Z(\gamma),C)&=\underset{0\leq \Gamma\in Z(\gamma)}{inf}\nu(\Gamma, C)\\
&= \underset{0\leq T\in \gamma}{inf}\nu(T-N(\gamma), C)\\
&=\nu(\gamma,C)-\nu(\gamma,C)=0.
\end{align}
\end{proof}

\begin{lemma}
\label{lemma neg eq}
Let $\gamma\in \mathcal{N}^\circ$, then $\{N(\gamma)\}=\{N(\{N(\gamma)\})\}$.
\end{lemma}

\begin{proof}
By the definition of $Z(\cdot)$, it is easy to see that $Z(\gamma_1 +\gamma_2)-Z(\gamma_1)-Z(\gamma_2)\in \mathcal{N}$ for any two $\gamma_1, \gamma_2\in \mathcal{N}$. In particular, we have $Z(\gamma )-Z(Z(\gamma))-Z(\{N(\gamma)\})\in \mathcal{N}$. Now Lemma \ref{lemma big zariski projection} implies $N(Z(\gamma))=\sum \nu(Z(\gamma), C)[C]=0$, so we have $Z(Z(\gamma))=Z(\gamma)-\{N(Z(\gamma))\}=Z(\gamma)$. And this yields $Z(\{N(\gamma)\})=0$, which is equivalent to the equality $\{N(\gamma)\}=\{N(\{N(\gamma)\})\}$.
\end{proof}

Now we can finish the proof of claim ($\ast$). Firstly, by the definition of $N(\{N(\gamma)\})$ and Siu decomposition, we have $N(\gamma)\geq N(\{N(\gamma)\})$. As Lemma \ref{lemma neg eq} shows they lie in the same Bott-Chern class, we must have $N(\gamma)= N(\{N(\gamma)\})$. For any positive current $T\in \{N(\gamma)\}$, using Siu decomposition and the definition of $N(\{N(\gamma)\})$ again, we have
$$T\geq\sum \nu(T, C)[C]\geq N(\{N(\gamma)\})=N(\gamma).$$
Thus $T=N(\gamma)$, and
$N(\gamma)$ is the unique positive current in the class $\{N(\gamma)\}$.

Next we show $N(\gamma)$ is an effective curve, that is, it is a finite sum of irreducible curves. Indeed, we will show $N(\gamma)$ is a sum of at most $\rho= dim_{\mathbb{R}} N_1 (X, \mathbb{R})$ irreducible curves. This follows from the following lemma.

\begin{lemma}
\label{lemma bouck injective }(compare with Proposition 3.11 of \cite{Bou04})
Let $\gamma\in \mathcal{N}^\circ$, and let $S$ the set of irreducible curves $C$ satisfying $\nu(\gamma, C)>0$, then $\# S \leq \rho$.
\end{lemma}

\begin{proof}
Take finite curves $C_1,...,C_k\in S$, and let $\Gamma=\sum_{i=1}^k a_i [C_i]$ with $a_i\in \mathbb{R}$. We claim that if the class $\{\Gamma\}=0$ then all $a_i =0$. This of course yields $\# S \leq \rho$. Write $\Gamma=\Gamma_+ - \Gamma_-$ such that both $\Gamma_+$ and $\Gamma_-$ are positive. Since we have assumed $\{\Gamma\}=0$, we have $N(\{\Gamma_+\})=N(\{\Gamma_-\})$. By the definition of $N(\gamma)$, we can take a positive constant $c$ large enough such that $\{cN(\gamma)-\Gamma_+\} \in \mathcal{N}$. By Lemma \ref{lemma neg eq} we know $Z(\{cN(\gamma)\})=cZ(\{N(\gamma)\})=0$, which implies $Z(\{\Gamma_+\})=0$. So we have $\{\Gamma_+\}=\{N(\{\Gamma_+\})\}$, and this implies $\Gamma_+=N(\{\Gamma_+\})$. This also holds for $\Gamma_-$. Combining with $N(\{\Gamma_+\})=N(\{\Gamma_-\})$, we get $\Gamma=\Gamma_+ - \Gamma_-=0$, which proves our claim.
\end{proof}

Finally let us prove Zariski projection preserves $\widehat{vol}_{\mathcal{N}}$. By the decomposition developed in \cite{LX15}, we know $\gamma\in \mathcal{N}^\circ$ can be uniquely decomposed as following:
\begin{align}
\label{eq L-X decomposition}
\gamma= B_\gamma ^{n-1}+ \zeta_\gamma
\end{align}
with $B_\gamma$ nef big, $\langle B_\gamma, \zeta_\gamma\rangle=0$, $\widehat{vol}_{\mathcal{N}}(B_\gamma ^{n-1})=\widehat{vol}_{\mathcal{N}}(\gamma)$ and $\zeta_\gamma\in \partial \mathcal{N}$. Denote by the same symbol $B_\gamma$ a smooth $(1,1)$-form in the class $B_\gamma$. Since $B_\gamma$ is nef, for any $\varepsilon>0$ there exists a smooth function $\psi_\varepsilon$ such that $B_\gamma+ \varepsilon \omega+ i\partial\bar \partial \psi_\varepsilon>0$. From this, it is easy to see for any $\varepsilon>0$ there exists a smooth $(n-2, n-2)$-form $\Psi_\varepsilon$ such that $$\Omega_\varepsilon:=B_\gamma^{n-1}+i\partial\bar \partial \Psi_\varepsilon \geq -\varepsilon \omega^{n-1}.$$
Denote by $T_\gamma$ a positive $(n-1,n-1)$-current in the class $\zeta_\gamma$, then $\Omega_\varepsilon+T_\gamma\in \gamma[-\varepsilon \omega^{n-1}]$. And by the definition of minimal multiplicity (see Definition \ref{def minimal multi}), we get
\begin{align}
\nu(\gamma, x)&=\underset{\varepsilon>0}{sup}\ \underset{T_\varepsilon}{inf}\nu(T_\varepsilon, x)\\
&\leq \underset{\varepsilon>0}{sup}\ \nu(\Omega_\varepsilon+T_\gamma,x)\\
&=\nu(T_\gamma,x).
\end{align}
The last line follows because $\Omega_\varepsilon$ is smooth. By Siu decomposition, the above inequality implies $\zeta_\gamma-\{N(\gamma)\}\in \mathcal{N}$. Thus $Z(\gamma)-B_\gamma ^{n-1}=\zeta_\gamma-\{N(\gamma)\}\in \mathcal{N}$, which yields
$$
\widehat{vol}_{\mathcal{N}}(Z(\gamma))\geq \widehat{vol}_{\mathcal{N}}(B_\gamma ^{n-1}).
$$
Combining with $\widehat{vol}_{\mathcal{N}}(\gamma)\geq \widehat{vol}_{\mathcal{N}}(Z(\gamma))$ and $\widehat{vol}_{\mathcal{N}}(\gamma)=\widehat{vol}_{\mathcal{N}}(B_\gamma ^{n-1})$, we finish the proof of the equality $\widehat{vol}_{\mathcal{N}}(\gamma)=\widehat{vol}_{\mathcal{N}}(Z(\gamma))$.
\end{proof}

\begin{remark}
\label{rmk remove big}
It will be interesting to know whether the statement for $N(\gamma)$ in Theorem \ref{thm zariski} is still true for $\gamma\in \partial\mathcal{N}$. Our above arguments show that the assumption $\gamma\in \mathcal{N}^\circ$ is important in Lemma \ref{lemma multiplicity}. And we need Lemma \ref{lemma multiplicity} to prove the other lemmas.
\end{remark}

\begin{remark}
\label{rmk z proj not mov}
One may expect that $Z(\gamma)$ could be represented by some positive smooth $(n-1,n-1)$-form, more precisely, one may expect $Z(\gamma)\in \overline{\mathcal{M}}$. Thus, by Proposition \ref{prop smoothing gauduchon} and Proposition \ref{prop smoothing balanced}, there exists a smooth positive $(n-1,n-1)$-form in the class $Z(\gamma)$ if $Z(\gamma)$ is an interior point of $\overline{\mathcal{M}}$. However, in general, $Z(\gamma)$ could not be a movable class.
Let $\pi: X\rightarrow \mathbb{P}^3$ be the blow-up along a point, and let $E= \mathbb{P}^2$ be the exceptional divisor. Let $\omega_{FS}$ be the Fubini-Study metric of $\mathbb{P}^3$ and let $\mathbb{P}^1 \subseteq E$ be a line of $E$, then we claim that
$$
Z(\{\pi^* (\omega_{FS} ^2) + [\mathbb{P}^1]\})= \{\pi^* (\omega_{FS} ^2 )+ [\mathbb{P}^1]\} \in \mathcal{N}^\circ \setminus \overline{\mathcal{M}}.$$
Firstly, it is easy to see $\{\pi^* (\omega_{FS} ^2 )+ [\mathbb{P}^1]\} \in \mathcal{N}^\circ$ which of course implies $Z(\{\pi^* (\omega_{FS} ^2) + [\mathbb{P}^1]\})\in \mathcal{N}^\circ$. For any point $x$, we can always choose an integration current in the class $[\mathbb{P}^1]$ but with its support avoiding $x$. Then we have
$\nu(\{\pi^* (\omega_{FS} ^2 )+ [\mathbb{P}^1]\}, x)=0$, which yields the equality
$$Z(\{\pi^* (\omega_{FS} ^2) + [\mathbb{P}^1]\})= \{\pi^* (\omega_{FS} ^2 )+ [\mathbb{P}^1]\}.
$$
Since we have $\{\pi^* (\omega_{FS} ^2 )+ [\mathbb{P}^1]\}\cdot E=-1$, the class $\{\pi^* (\omega_{FS} ^2 )+ [\mathbb{P}^1]\}$ can not be movable. Comparing with the Zariski decompositions developed in \cite{FL13} and \cite{LX15}, $Z(\gamma)$ not always being movable is its disadvantage in many applications.
Anyhow, if $\gamma\in \mathcal{N}^ \circ$, then Lemma \ref{lemma multiplicity} and Lemma \ref{lemma big zariski projection} show that we can always choose a positive current in the class $Z(\gamma)$ with its Lelong number along any curve being arbitrarily small. In some sense, this means that $Z(\gamma)$ is less singular than $\gamma$.
Indeed, $Z(\gamma)\in \overline{\mathcal{K}}$ if $X$ is a K\"ahler surface.
\end{remark}

At the end of this section, we show that Zariski decomposition for 1-cycles is trivial for compact K\"ahler manifold with nef tangent bundle.

\begin{proposition}
\label{prop nef TX}
Let $X$ be a compact K\"ahler manifold with nef tangent bundle, then $\gamma=Z(\gamma)$ for any $\gamma\in \mathcal{N}$. Indeed, we will have $\gamma=Z(\gamma) \in \overline{\mathcal{M}}$.
\end{proposition}

\begin{proof}
This follows from Demailly's regularization theorem of positive $(1,1)$-currents and our previous work on transcendental holomorphic Morse inequality.

If $TX$ is nef, then $\overline{\mathcal{K}}=\mathcal{E}$ (see Corollary 1.5 of \cite{Dem92}). Now let $\alpha, \beta\in \overline{\mathcal{K}}$ be two nef classes such that $\alpha^n -n \alpha^{n-1}\cdot \beta>0$, then $\alpha-\beta$ must be an interior point of $\mathcal{E}$ (see \cite{Xia13}, \cite{Pop14}). By $\overline{\mathcal{K}}=\mathcal{E}$,  $\alpha-\beta$ must be a K\"ahler class. In particular, $\alpha-t\beta \in \overline{\mathcal{K}}$ for $t\in [0,1]$.

Consider the difference $vol(\alpha-\beta)-vol(\alpha)$, we have
\begin{align*}
\label{eq rmk morse}
vol(\alpha-\beta)-vol(\alpha)&=\int_{0} ^1 \frac{d}{dt}vol(\alpha-t\beta)dt\\
&=\int_{0} ^1 -n(\alpha-t\beta)^{n-1} \cdot \beta dt\\
&\geq -n\alpha^{n-1} \cdot \beta,
\end{align*}
thus $vol(\alpha-\beta)\geq \alpha^n -n\alpha^{n-1} \cdot \beta$. Using the same arguments as \cite{BDPP13}, this of course implies the cone duality $\mathcal{E}^ \vee =\overline{\mathcal{M}}$. And this yields $\mathcal{E}^ \vee =\overline{\mathcal{M}}=\overline{\mathcal{B}}$ (see e.g. \cite{FX14}). Using $\overline{\mathcal{K}}=\mathcal{E}$ again, $\overline{\mathcal{K}}^ \vee = \mathcal{N}$ implies $\mathcal{N}=\overline{\mathcal{B}}$. Since $\gamma\in \mathcal{N}=\overline{\mathcal{B}}$, for any $\varepsilon>0$ there exists a smooth $(n-1,n-1)$-form $\Omega_\varepsilon \in \gamma$ such that $\Omega_\varepsilon\geq -\varepsilon \omega^{n-1}$. Now by the definition of minimal multiplicity (see Definition \ref{def minimal multi}), we get $\nu(\gamma,x)=0$ for every point, yielding $N(\gamma)=0$. This implies $\gamma= Z(\gamma)$.
\end{proof}

\section{Further discussions}

\subsection{Another invariant of movable class}

As remarked in the previous section, under the assumption of the conjecture on transcendental holomorphic Morse inequality, we will have $\overline{\mathcal{M}}=\overline{\mathcal{G}}=\overline{\mathcal{B}}$. And invariant of Gauduchon or balanced classes would be invariant of movable class.
Inspired by our previous work \cite{FX14}, we introduce another invariant $\mathfrak{M}_{CY}$ of Gauduchon class by using form-type Calabi-Yau equations (or complex Monge-Amp\`{e}re equations for $(n-1)$-plurisubharmonic functions) (see \cite{FWW10}, \cite{TW13a}, \cite{TW13b}).

\begin{definition}
\label{def vol_B G}
Let $X$ be an $n$-dimensional compact K\"ahler manifold, and let $\gamma$ be a Gauduchon class. Then we define $\mathfrak{M}_{CY}(\gamma)$ as following:
$$
\mathfrak{M}_{CY}(\gamma):=\underset{\Phi, \omega}{sup}
\ \{c_{\Phi, \omega}\}
$$
where
$c_{\Phi, \omega}$ is a positive constant satisfying $\omega^n = c_{\Phi, \omega} \Phi$ such that $\Phi$ is a smooth volume form with $\int \Phi =1$ and $\omega^{n-1}\in \gamma$ is a Gauduchon metric.
\end{definition}

Assume $\gamma=\alpha^{n-1}$ for some $\alpha\in {\mathcal{K}}$, we prove that $\mathfrak{M}_{CY}(\gamma)=vol(\alpha)$.

\begin{proposition}
\label{prop vol_B G coincide nef}
Let $X$ be an $n$-dimensional compact K\"ahler manifold, and let $\gamma=\alpha^{n-1}$ for some K\"ahler class $\alpha$. Then we have $\mathfrak{M}_{CY}(\gamma)=vol(\alpha)$.
\end{proposition}

\begin{proof}
Firstly, since $\alpha$ is a K\"ahler class, by Calabi-Yau theorem (see \cite{Yau78}), there exists an unique K\"ahler metric $\alpha_u\in \alpha$ such that $\alpha_u ^n =vol(\alpha)\Phi$. In particular, $c_{\Phi, \alpha_u}=vol(\alpha)$, thus $vol(\alpha)\leq \mathfrak{M}_{CY}(\gamma)$. We claim that, for any $\Phi, \omega$ in the definition of $\mathfrak{M}_{CY}(\gamma)$, we have $$c_{\Phi, \omega}\leq vol(\alpha).$$
For any fixed such $\Phi, \omega$, we first apply Calabi-Yau theorem to find a K\"ahler metric $\alpha_\psi$ such that $$\alpha_\psi ^ n =\frac{vol(\alpha)}{c_{\Phi,\omega}}\omega^n.$$
Using the following pointwise inequality $$\omega^{n-1}\wedge \alpha_\psi \geq
(\frac{\alpha_\psi^n}{\omega^n})^{\frac{1}{n}} \omega^n$$
and $\omega^{n-1}\in \gamma=\alpha^{n-1}$ being Gauduchon, we estimate $vol(\alpha)$ as following:
$$
vol(\alpha)=\int \gamma\wedge \alpha
=\int \omega^{n-1}\wedge \alpha_\psi
\geq\int (\frac{\alpha_\psi^n}{\omega^n})^{\frac{1}{n}} \omega^n
= vol(\alpha)^{\frac{1}{n}}c_{\Phi, \omega} ^{\frac{n-1}{n}}.
$$
This of course implies $\mathfrak{M}_{CY}(\gamma)\leq vol(\alpha)$. Combining with $vol(\alpha)\leq \mathfrak{M}_{CY}(\gamma)$, we get the desired equality $\mathfrak{M}_{CY}(\gamma)=vol(\alpha)$.
\end{proof}

Note that $\mathfrak{M}_{CY}$ is an analytical invariant by solving non-linear PDEs, and $\mathfrak{M}$ is an intersection-theoretic invariant. It will be very interesting to compare $\mathfrak{M}_{CY}$ and $\mathfrak{M}$, and we have the following proposition.

\begin{proposition}
\label{prop compare vol_. vol}
Let $X$ be an $n$-dimensional compact K\"ahler manifold, and let $\gamma$ be a Gauduchon class. Then we always have $\mathfrak{M}_{CY}(\gamma)\leq \mathfrak{M}(\gamma)$. Moreover, they coincide over K\"ahler classes, that is, $\mathfrak{M}_{CY}(\alpha^{n-1})=\mathfrak{M}(\alpha^{n-1})$ for any K\"ahler class $\alpha$.
\end{proposition}

\begin{proof}
For any smooth volume form $\Phi$ with $\int \Phi=1$ and any $\beta\in \mathcal{E}$ with $vol(\beta)=1$, by the singular version of Calabi-Yau theorem (see \cite{Bou02a}), there exists a positive $(1,1)$-current $T\in \beta$ such that $T_{ac}^n = \Phi$ almost everywhere. Now for any Gauduchon metric $\omega^{n-1}\in \gamma$ in the definition of $c_{\Phi, \omega}$, we get
 $$\langle\beta, \gamma\rangle=\int T\wedge \omega^{n-1}\geq
 \int T_{ac} \wedge \omega^{n-1}\geq \int (\frac{T_{ac}^n}{\Phi})^{\frac{1}{n}}
 (\frac{\omega^n}{\Phi})^{\frac{n-1}{n}}\Phi =c_{\Phi,\omega}^{\frac{n-1}{n}}.$$
Since $\beta$, $\omega^{n-1}$ and $\Phi$ are (conditionally) arbitrary, we get $\mathfrak{M}_{CY}(\gamma)\leq \mathfrak{M}(\gamma)$.

By Proposition \ref{prop vol hat coin kahler} and Proposition \ref{prop vol_B G coincide nef}, we have $\mathfrak{M}_{CY}(\alpha^{n-1})=\mathfrak{M}(\alpha^{n-1})$ for any K\"ahler class $\alpha$.
\end{proof}

\begin{remark}
The above proposition also implies that $\mathfrak{M}_{CY}(\gamma)$ is always well defined over compact K\"ahler manifolds, that is, $\mathfrak{M}_{CY}(\gamma)<\infty$, which is not explicit from its definition.
\end{remark}

\begin{remark}
Let $X$ be an $n$-dimensional compact K\"ahler manifold, \emph{we do not know whether $\mathfrak{M}_{CY}(\gamma)= \mathfrak{M}(\gamma)$ for any $\gamma\in \mathcal{G}$}. As we always have $\mathfrak{M}_{CY}(\gamma)\leq \mathfrak{M}(\gamma)$, we only need to show $\mathfrak{M}_{CY}(\gamma)\geq \mathfrak{M}(\gamma)$.
We also want to know the behaviour of $\mathfrak{M}_{CY}$
under bimeromorphic maps (compare with Proposition \ref{prop modification vol_ G B}). In particular, \emph{we do not know whether we have $\mathfrak{M}_{CY}(\mu_* \tilde{\gamma})\geq \mathfrak{M}_{CY}(\tilde{\gamma})$}. If this would be true, then we can use this invariant in Theorem \ref{thm characterize vol of divisor}. It will also be very interesting to study the concavity of $\mathfrak{M}_{CY}$. To study these problems, we need know more about the family of constants $c_{\Phi, \omega}$ in the definition of $\mathfrak{M}_{CY}$.
\end{remark}

\subsection{A general approach}

This section comes from a suggestion of Mattias Jonsson. Let $\mathcal{C}\subseteq V$ be a proper convex cone of a real vector space. Let $u: \overline{\mathcal{C}}\rightarrow \mathbb{R}_{+} $ be a continuous function. Let $p>1$ be a constant.
Let $\mathcal{C}^ \vee \subseteq V^*$ be the dual of $\mathcal{C}$. In general, we can define the dual of $u$ in the following way:
\begin{align*}
\widehat{u}(x^*):=\underset{y\in \mathcal{C}_1}{inf}\langle x^*, y\rangle^q,
\end{align*}
where $\mathcal{C}_1=\{y\in \mathcal{C}|\ u(y)=1\}$ and $\frac{1}{p}+\frac{1}{q}=1$. This is similar to some kind of Legendre-Fenchel transform. It is easy to see $\widehat{u}^{\frac{1}{q}}$ is concave and homogeneous of degree one over $\mathcal{C}^ \vee$. If we assume $u^{1/p}$ is concave and homogeneous of degree one. Then we have
\begin{align*}
\widehat{u}(x^*):=\underset{y\in \mathcal{C}_{\geq 1}}{inf}\langle x^*, y\rangle^q,
\end{align*}
where $\mathcal{C}_{\geq 1}=\{y\in \mathcal{C}|\ u(y)\geq 1\}$. Since $u^{1/p}$ is concave, $\mathcal{C}_{\geq 1}$ is a convex closed subset of $\mathcal{C}$.

In our definition of $\widehat{vol}_\mathcal{N}$ for 1-cycles over compact K\"ahler manifold, we have $\mathcal{C}=\mathcal{K}$, $u=vol$ and $p=1/n$. For $1<k<n-1$, let $\mathcal{N}_k \in H_{BC}^{k,k}(X, \mathbb{R})$ be the cone generated by $d$-closed positive $(k,k)$-currents. It will be interesting if one can generalize this kind of construction of volume to $\mathcal{N}_k$, thus define a volume functional for general $k$-cycles. Principally, we first need to define a function $u$ on some kind of smooth positive $(n-k, n-k)$-forms. However, unlike the case for the cone $\mathcal{N}$, the structure of the dual of $\mathcal{N}_k$ is not clear (and indeed this problem is still widely open). As a starting point, it will be very interesting to carry out the above general approach over toric varieties.
\\

\noindent
\textbf{Acknowledgements}: I would like to thank Professor Jean-Pierre Demailly for suggesting this topic, stimulating discussions and helpful suggestions. Thanks also for Professor Mattias Jonsson and Brian Lehmann for many useful discussions. I also would like to thank Professor Jixiang Fu for constant encouragement. This work is supported by China Scholarship Council.

\noindent
\textsc{Institut Fourier, Universit\'{e} Joseph Fourier-Grenoble I, 38402 Saint-Martin d'H\`{e}res, France} \\
\textsc{and}\\
\textsc{Institute of Mathematics, Fudan University, 200433 Shanghai, China}\\
\verb"Email: jian.xiao@ujf-grenoble.fr"\\


\begin{thebibliography}{99}
\bibitem[AB95]{AB95}
L. Alessandrini and G. Bassanelli, {\sl
Modifications of compact balanced manifolds}, C. R. Math. Acad. Sci. Paris {\bf 320} (1995), 1517-1522.

\bibitem[AT13]{AT13}
D. Angella, A. Tomassini, {\sl On the $\partial\bar\partial$-Lemma and Bott-Chern cohomology}, Invent. Math. {\bf 192} (2013), 71-81.

\bibitem[BDPP13]{BDPP13}
S. Boucksom, J.-P. Demailly, M. Paun, T. Peternell, {\sl The pseudo-effective cone of a compact K\"ahler manifold and varieties of negative Kodaira dimension}, J. Algebraic Geom. {\bf 22} (2013) 201-248.


\bibitem[Bou02a]{Bou02a}
S. Boucksom, {\sl On the volume of a line bundle}, Internat. J. Math. {\bf 13} (2002), 1043-1063.

\bibitem[Bou02b]{Bou02b}
S. Boucksom, {\sl Cones positifs des vari\'{e}t\'es complexes compactes}, Phd Thesis, Institut Fourier, 2002.

\bibitem[Bou04]{Bou04}
S. Boucksom, {\sl Divisorial Zariski decompositions on compact complex manifolds}, Ann. Sci. \'{E}cole Norm. Sup. {\bf 37} (2004), 45-76.

\bibitem[Cut86]{Cut86}
S. D. Cutkosky, {\sl Zariski decomposition of divisors on algebraic varieties}, Duke Math. J. {\bf 53} (1986), 149-156.

\bibitem[DELV11]{DELV11}
O. Debarre, L. Ein, R. Lazarsfeld, C. Voisin,
{\sl Pseudoeffective and nef
classes on abelian varieties}, Compos. Math. {\bf 147} (2011),  1793-1818.

\bibitem[Dem92]{Dem92}
J.-P. Demailly, {\sl Regularization of closed positive currents and intersection theory}, J. Algebraic Geom. {\bf 1} (1992), 361-409.

\bibitem[Dem10]{Dem10}
J.-P. Demailly, {\sl A converse to the Andreotti-Grauert theorem},  arXiv preprint, 2010.

\bibitem[DP04]{DP04}
J.-P. Demailly, M. Paun, {\sl Numerical characterization of the K\"ahler cone of a compact K\"ahler manifold},
Ann. Math. {\bf 159} (2004), 1247-1274.

\bibitem[FX14]{FX14}
J. Fu, J. Xiao, {\sl Relations between the K\"ahler cone and the balanced cone of a K\"ahler manifold}, Adv. Math. {\bf 263} (2014), 230-252.

\bibitem[FL13]{FL13}
M. Fulger, B. Lehmann, {\sl Zariski decompositions of numerical cycle classes}, arXiv preprint, 2013.

\bibitem[FWW10]{FWW10}
J. Fu, Z. Wang, D. Wu, {\sl Form-type Calabi-Yau equations}, Math. Res. Lett.
{\bf 17} (2010), 887-903.
\bibitem[Lam99]{Lam99}
A. Lamari, {\sl Courants k\"ahl\'eriens et surfaces compactes}, Ann. Inst. Fourier (Grenoble) {\bf 49} (1999), 263-285.

\bibitem[Leh13]{Leh13}
B. Lehmann, {\sl Geometric characterizations of big cycles},
arXiv preprint, 2013.

\bibitem[LX15]{LX15}
B. Lehmann, J. Xiao, {\sl Volume and mobility}, in preparation, 2015.

\bibitem[Nak04]{Nak04}
N. Nakayama, {\sl Zariski-decomposition and abundance}, MSJ Mem., vol. {\bf 14}, Math. Soc. Japan, Tokyo, 2004.

\bibitem[Pop14]{Pop14}
D. Popovici, {\sl An observation relative to a paper by J. Xiao},
arXiv preprint, 2014.

\bibitem[Siu74]{Siu74}
Y.-T. Siu, {\sl Analyticity of sets associated to Lelong numbers and the extension of closed positive currents}, Invent. Math. {\bf 27} (1974), 53-156.

\bibitem[Sul76]{Sul76}
D. Sullivan, {\sl Cycles for the danymical study of foliated manifolds and complex manifolds}, Invent. Math. {\bf 36} (1976), 225-255.

\bibitem[Tom10]{Tom10}
M. Toma, \emph{\sl A note on the cone of mobile curves},
C. R. Math. Acad. Sci. Paris {\bf 348} (2010), 71-73.

\bibitem[TW13a]{TW13a}
V. Tosatti and B. Weinkove, {\sl The Monge-Amp\`ere equation for (n-1)-plurisubharmonic functions on a compact K\"ahler manifold},
arXiv preprint, 2013.

\bibitem[TW13b]{TW13b}
V. Tosatti and B. Weinkove, {\sl Hermitian metrics, $(n-1,n-1)$ forms and Monge-Amp\`{e}re equations}, arXiv preprint, 2013.

\bibitem[Xia13]{Xia13}
J. Xiao, {\sl Weak transcendental holomorphic Morse inequalities on compact K\"ahler manifolds}, arXiv preprint 2013, to appear in Ann. Inst. Fourier (Grenoble).

\bibitem[Yau78]{Yau78}
S.-T. Yau, {\sl On the Ricci curvature of a compact
K\"ahler manifold and the complex Monge-Amp\`{e}re equation}, I, Comm.
Pure. Appl. Math. {\bf 31} (1978), 339-411.
\end{thebibliography}
\end{document}